\documentclass[a4paper]{amsart}

\usepackage[marginparwidth=0pt,margin=20truemm]{geometry} % margin
\usepackage{amsmath,amssymb,amsthm,amsfonts,amscd,mathrsfs,mathtools}
\usepackage{ascmac}
\usepackage{graphicx,xcolor}
\usepackage{tikz}
\usetikzlibrary{arrows,arrows.meta,positioning,fit,calc,intersections,through,backgrounds,quotes,angles,cd}

\usepackage{thmtools}
\usepackage{stmaryrd}
\usepackage{braket}
\usepackage{enumerate}

\theoremstyle{plain}
\newtheorem{thm}{Theorem}[section]
\newtheorem*{thm*}{Theorem}
\newtheorem{lem}[thm]{Lemma}
\newtheorem{prop}[thm]{Proposition}
\newtheorem{cor}[thm]{Corollary}

\theoremstyle{definition}

\newtheorem{dfn}[thm]{Definition}

\theoremstyle{remark}
\newtheorem{rem}[thm]{Remark}

\numberwithin{equation}{section}
\AtEndEnvironment{rem}{\unskip\nobreak\hfill$\square$}
\AtEndEnvironment{dfn}{\unskip\nobreak\hfill$\square$}
\providecommand{\texorpdfstring}[2]{#1}

\makeatletter
\AtBeginDocument{
}
\makeatother

\newlength{\tocSectionExtraIndent}
\newlength{\tocSubsectionExtraIndent}
\newlength{\tocSubsubsectionExtraIndent}
\makeatletter
\def\@tocline#1#2#3#4#5#6#7{\relax
  \ifnum #1>\c@tocdepth \else
    \par \addpenalty\@secpenalty \addvspace{#2}%
    \begingroup \hyphenpenalty\@M
    \@ifempty{#4}{\@tempdima\csname r@tocindent\number#1\endcsname}{\@tempdima#4}%
    \parindent\z@ \leftskip#3\relax \advance\leftskip\@tempdima\relax
    \rightskip\@pnumwidth plus4em \parfillskip-\@pnumwidth
    #5\leavevmode\hskip-\@tempdima
    \ifcase #1
      \or \hskip \tocSectionExtraIndent
      \or \hskip \tocSubsectionExtraIndent
      \or \hskip \tocSubsubsectionExtraIndent
      \else \hskip 3em
    \fi
    #6\nobreak\relax
    \hfill\hbox to\@pnumwidth{\@tocpagenum{#7}}\par
    \nobreak
    \endgroup
  \fi
}
\makeatother

\usepackage{hyperref}
\hypersetup{
  colorlinks=false,        % 文字色は変えず「枠」で表示
  pdfborder={0 0 1},       % 枠の太さ（最後の 1 を 2,3... にすると太くなる）
  linkbordercolor={1 0 0}, % 内部リンク（\ref, 目次など）の枠=赤
  citebordercolor={1 0 0}, % \cite の枠=赤
  urlbordercolor={1 0 0},  % URL の枠=赤
}

\newcommand{\lcap}{\mathbin{\overset{\leftarrow}{\cap}}}
\newcommand{\rcap}{\mathbin{\overset{\rightarrow}{\cap}}}

\newcommand{\lcup}{\mathbin{\overset{\leftarrow}{\cup}}}
\newcommand{\rcup}{\mathbin{\overset{\rightarrow}{\cup}}}

\title{A Semiclassical Limit of Large-Colour State Sums\\ for Semisimple Lie Algebras}
\author{Atsuhide Nagasaka}

\address{Graduate School of Mathematical Sciences, The University of Tokyo,
3-8-1 Komaba, Meguro-ku, Tokyo 153-8914, Japan}
\email{atsunaga@ms.u-tokyo.ac.jp, atsun.naga.524@outlook.jp}


\makeatletter
\patchcmd{\@maketitle}{\topskip42\p@}{\topskip28\p@}{}{}
\makeatother

\begin{document}

\begin{abstract}
    The Melvin--Morton--Rozansky conjecture relates the large-colour semiclassical limit of the coloured Jones polynomial to the inverse Alexander polynomial. Bar-Natan and Garoufalidis~\cite{bar1996melvin} proved the conjecture using weight systems and, more generally, established the corresponding formula for arbitrary complex semisimple Lie algebras.

We give a direct proof of this semisimple Lie algebra generalization by analysing a state sum. For a complex semisimple Lie algebra $\mathfrak{g}$ and a dominant integral weight $\lambda$, we use the embedding $V_{d\lambda}\hookrightarrow V_{\lambda}^{\otimes d}$ to analyse the state sum. Under the specialization $q=e^{h/d}$, we show that, as $d\to\infty$, only the identity terms and those proportional to $E_{\alpha}\otimes F_{\alpha}$ contribute at leading order. It follows that the leading-order contribution decomposes over the positive roots, with each root contribution evaluated as in the $\mathfrak{sl}_{2}$ case. Writing $t_{\alpha}=e^{-2(\lambda,\alpha)h}$, the limit is therefore given by the product of $\Delta_{K}(t_{\alpha})^{-1}$ over the positive roots $\alpha$.
\end{abstract}

\maketitle

\tableofcontents

\section{Introduction}
\label{sec:Intro}

Let $K$ be a knot. The coloured Jones polynomial of $K$ and its Alexander polynomial $\Delta_{K}$ arise from rather different constructions, yet the Melvin--Morton--Rozansky (MMR) conjecture reveals a remarkable relation between them. Melvin and Morton conjectured that the leading part of the large-colour expansion of the coloured Jones polynomial is governed by the inverse Alexander polynomial \cite{melvin1995coloured}. In the normalization used in this paper, this relation takes the form
\[
    \lim_{d\to\infty}
    \left.
        \frac{J_{\mathfrak{sl}_{2}}(K,V_{d+1})}{[d+1]_q}
    \right|_{q=e^{h/d}}
    = \Delta_K(e^{-2h})^{-1},
\]
where the limit is taken coefficientwise in $h$.

A rigorous proof was given by Bar-Natan and Garoufalidis \cite{bar1996melvin} using weight systems. Remarkably, their method also establishes the corresponding formula for arbitrary complex semisimple Lie algebras. 
Let $\mathfrak{g}$ be a complex semisimple Lie algebra and fix a set $\Delta_{+}$  of positive roots, with the invariant bilinear form $(\cdot,\cdot)$ normalized so that every short root has squared length $2$.
With the normalization used in this paper and the specialization $q=e^{h/d}$, the general formula takes the following form (see Theorem~\ref{thm:mmr-semisimple} below):
\begin{thm*}
    For every dominant integral weight $\lambda$,
    \[
        \lim_{d\to\infty}
        \left.
        \frac{J_{\mathfrak{g}}(K,V_{d\lambda})}
             {\operatorname{dim}_{q}(V_{d\lambda})}
        \right|_{q=e^{h/d}}
        = \prod_{\alpha\in\Delta_{+}} \Delta_{K}\!\left(e^{-2(\lambda,\alpha)h}\right)^{-1}.
    \]
    The limit is coefficientwise in $h$.
\end{thm*}

We give a direct proof of Theorem~\ref{thm:mmr-semisimple} by analysing a state sum for the Reshetikhin--Turaev invariant in the regime $q=e^{h/d}$, without passing to weight systems.

There is also a physical motivation for this relation in the $\mathfrak{sl}_{2}$ case. Witten interpreted the Jones polynomial as a Wilson-loop expectation value in $SU(2)$ Chern--Simons theory \cite{witten1989quantum}. Starting from this picture, Rozansky studied the perturbative contribution of the trivial connection and related the corresponding torsion to the Alexander polynomial \cite{rozansky1996contribution}. In this way, the perturbative expansion of Chern--Simons theory suggests the appearance of the Alexander polynomial in this limit.

Several approaches have given quantum or combinatorial descriptions of the relation in the $\mathfrak{sl}_{2}$ case. 
Rozansky later analysed the universal $\mathcal{R}$--matrix by regarding it as a perturbation of the Burau representation \cite{rozansky1998universalr}. 
Lin and Wang interpreted a cabling state sum for the coloured Jones polynomial in terms of random walks on a knot diagram, with the inverse Alexander polynomial appearing as a graph zeta function \cite{lin2001random}. 
Garoufalidis and Loebl expressed the coloured Jones function as an evaluation of the inverse of a noncommutative fermionic partition function, again making the relation with the Alexander polynomial visible directly in the formula \cite{garoufalidis2006non}.

The higher-rank case presents an additional difficulty. For $\mathfrak{sl}_2$ there is only one positive root. 
For a general complex semisimple Lie algebra, the universal $\mathcal{R}$--matrix involves an ordered product indexed by the positive roots, in which the factor corresponding to $\alpha$ has summands proportional to $E_{\alpha}^{n}\otimes F_{\alpha}^{n}$ for $n\geq 0$.
A direct state-sum argument must therefore determine which of these summands contribute at leading order and explain why the contributions associated with the positive roots factorize at leading order.

There have also been recent works on related higher-rank large-colour problems. Gruen constructed a large-colour $\mathcal{R}$--matrix for symmetric representations of $\mathfrak{sl}_N$ \cite{gruen2022mathfrak}, and Gruen and San Mart\'{\i}n Su\'arez subsequently extended this construction to arbitrary irreducible representations of $\mathfrak{sl}_{3}$, obtaining the invariant $F_{K}^{\mathfrak{sl}_{3}}$ for positive braid knots \cite{gruen2025large}.
In the latter case, the classical limit is the inverse of a product of three Alexander polynomials. Their result also contains higher-order terms in the perturbative expansion, whereas the present argument concerns only the semiclassical limit under the specialization $q=e^{h/d}$ and applies to an arbitrary complex semisimple Lie algebra, an arbitrary dominant integral weight, and an arbitrary knot.

%%%%%%%%%%%%%%%%%%%%%%%%%%%%%%%%%%%%%%%%%%%%%%%%%%%%%%%%%%%%%%%%%%%%%%%%%%%%%%%%%%%
% Section introduction
%%%%%%%%%%%%%%%%%%%%%%%%%%%%%%%%%%%%%%%%%%%%%%%%%%%%%%%%%%%%%%%%%%%%%%%%%%%%%%%%%%%

%%%%%%%%%%%%%%%%%%%%%%%%%%%%%%%%%%%%%%%%%%%%%%%%%%%%%%%%%%%%%%%%%%%%%%%%%%%%%%%%%%%
% \S 2
%%%%%%%%%%%%%%%%%%%%%%%%%%%%%%%%%%%%%%%%%%%%%%%%%%%%%%%%%%%%%%%%%%%%%%%%%%%%%%%%%%%
We begin in \S\ref{sec:zeta_alex} by recalling the graph-theoretic description of the Alexander polynomial used in the subsequent state-sum analysis. Following Lin and Wang \cite{lin2001random}, we associate a weighted directed graph $(K_{D^{\ast}}, w_{t})$ to a $(1,1)$--tangle diagram $D^{\ast}$ whose closure is $K$. Its graph zeta function satisfies
\[
    t^{-\frac{\operatorname{rot}(D^{\ast})}{2}}
    \zeta(K_{D^{\ast}},w_{t})
    = \Delta_{K}(t)^{-1},
\]
where $\operatorname{rot}(D^{\ast})$ denotes the rotation number; see \eqref{eq:alexander-zeta}.
By Lemma~\ref{lem:zeta_flow}, the graph zeta function has the flow expansion
\[
    \zeta(K_{D^{\ast}},w_{t})
    = \sum_{f\in\mathcal{F}(K_{D^{\ast}})} \operatorname{mult}(f)w_{t}(f).
\]
The subsequent state-sum analysis will show that the leading contribution associated with a fixed flow $f$ is $\operatorname{mult}(f)w_{t}(f)$.

%%%%%%%%%%%%%%%%%%%%%%%%%%%%%%%%%%%%%%%%%%%%%%%%%%%%%%%%%%%%%%%%%%%%%%%%%%%%%%%%%%%
% \S 3
%%%%%%%%%%%%%%%%%%%%%%%%%%%%%%%%%%%%%%%%%%%%%%%%%%%%%%%%%%%%%%%%%%%%%%%%%%%%%%%%%%%

In \S\ref{sec:sl2}, we first carry out this program for $\mathcal{U}_q(\mathfrak{sl}_2)$. Using the $\mathcal{U}_q(\mathfrak{sl}_2)$--module embedding
\[
    V_{d+1} \hookrightarrow V_2^{\otimes d},
\]
we realize the coloured Jones polynomial in our normalization as a matrix coefficient on $V_2^{\otimes d}$ and expand it as a state sum:
\[
    Z_d^{RT}
    := \frac{J_{\mathfrak{sl}_2}(K,V_{d+1})}{[d+1]_q} 
    = \sum_{\sigma\in S_d} W_d^{RT}(\sigma).
\]

We modify the coefficients contributed by the crossings, without changing which basis vectors occur with nonzero coefficients, and let $W_d^{\mathrm{lead}}(\sigma)$ denote the contribution obtained from these modified coefficients.
From now on, we set $q = e^{h/d}$ and $t = e^{-2h}$. Put
\[
    Z_d^{\mathrm{lead}}
    := \sum_{\sigma\in S_d} W_d^{\mathrm{lead}}(\sigma).
\]
The comparison with the original crossing coefficients, together with the scalar contributions from caps and cups, gives
\[
    Z_d^{RT}
    = t^{-\frac{\operatorname{rot}(D^{\ast})}{2}} Z_d^{\mathrm{lead}}
    + O(d^{-1})
\]
coefficientwise as $d\to\infty$.

We decompose $Z_d^{\mathrm{lead}}$ according to flows on the weighted directed graph of \S\ref{sec:zeta_alex}.
For a fixed flow $f$, let $r(f)$ denote the number of jumps and let $S_{d}(f)$ denote the corresponding subset of $S_{d}$.
The summands are organized by
\[
    a=(a_1,\ldots,a_{r(f)}), \qquad
    0\leq a_{k}\leq d-1.
\]
Lemma~\ref{lem:bad-set} shows that, for all but $O(d^{r(f)-1})$ choices of $a$, there are exactly $\operatorname{mult}(f)$ corresponding summands. Since the summands are $O(d^{-r(f)})$, the contribution from the remaining choices is $O(d^{-1})$.
Consequently, Proposition~\ref{prop:fixed-flow-evaluation} gives, for each fixed flow $f$,
\[
    \sum_{\sigma\in S_d(f)} W_d^{\mathrm{lead}}(\sigma)
    = \operatorname{mult}(f)w_t(f) + O(d^{-1})
\]
coefficientwise as $d\to\infty$.
For every fixed $m\geq 0$, only finitely many flows can contribute to coefficients of degree at most $m$. Summing over these flows and applying Lemma~\ref{lem:zeta_flow} therefore gives the graph zeta function of \S\ref{sec:zeta_alex}, and the normalization \eqref{eq:alexander-zeta} completes the proof of the Melvin--Morton--Rozansky conjecture.

%%%%%%%%%%%%%%%%%%%%%%%%%%%%%%%%%%%%%%%%%%%%%%%%%%%%%%%%%%%%%%%%%%%%%%%%%%%%%%%%%%%
% \S 4
%%%%%%%%%%%%%%%%%%%%%%%%%%%%%%%%%%%%%%%%%%%%%%%%%%%%%%%%%%%%%%%%%%%%%%%%%%%%%%%%%%%

In \S\ref{sec:semisimple}, we extend the argument to an arbitrary complex semisimple Lie algebra $\mathfrak{g}$ and a dominant integral weight $\lambda$. Using the $\mathcal{U}_q(\mathfrak{g})$--module embedding
\[
    V_{d\lambda} \hookrightarrow V_{\lambda}^{\otimes d},
\]
we realize the quantum invariant in our normalization as a matrix coefficient on $V_{\lambda}^{\otimes d}$ and use the ordered product expansion of the universal $\mathcal{R}$--matrix over the positive roots. The main point is to determine which summands of the resulting state sum contribute at leading order.
For $r(\eta)$ and $s(\eta)$ defined in \S\ref{sec:semisimple}, Lemmas~\ref{lem:semisimple-order} and \ref{lem:semisimple-index-count} show that the total contribution with
\[
    r(\eta)=r, \qquad
    s(\eta)=s
\]
is $O(d^{s-r})$ coefficientwise. Since $s(\eta)\leq r(\eta)$, for every fixed $m\geq 0$, the total contribution to coefficients of degree at most $m$ from summands with $s(\eta)<r(\eta)$ is $O(d^{-1})$.

Let $Z_{d,\lambda}^{\mathrm{lead}}$ denote the sum of the summands $W_{d,\lambda}^{\mathrm{cross}}(\eta)$ for which $s(\eta) = r(\eta)$.
Put
\[
    \Delta_{\lambda,+}
    :=
    \left\{
        \alpha\in\Delta_{+}
        \,\middle|\,
        \langle\lambda,\alpha^{\vee}\rangle>0
    \right\}, \qquad
    x_{\alpha}:=F_{\alpha}v_{\lambda}.
\]
Proposition~\ref{prop:semisimple-root-reduction} shows that, for every summand $W_{d,\lambda}^{\mathrm{cross}}(\eta)$ contributing to $Z_{d,\lambda}^{\mathrm{lead}}$ and every elementary factor $e$, either the identity summand is selected from every root factor in the ordered product over the positive roots, or there is a unique $\alpha \in \Delta_{\lambda,+}$ such that a non-identity summand proportional to $E_{\alpha}\otimes F_{\alpha}$ is selected from the root factor corresponding to $\alpha$, while the identity summand is selected from every other root factor. In the latter case, the transition at $e$ is
\[
    x_{\alpha}\otimes v_{\lambda}
    \longrightarrow v_{\lambda}\otimes x_{\alpha}.
\]
We then obtain $Z_{d,\lambda}^{A}$ from $Z_{d,\lambda}^{\mathrm{lead}}$ by modifying the coefficients in Lemma~\ref{lem:semisimple-local-root}.
Proposition~\ref{prop:semisimple-cartan-redistribution} gives
\[
    Z_{d,\lambda}^{\mathrm{lead}}
    = Z_{d,\lambda}^{A} + O(d^{-1})
\]
coefficientwise.
As in \S\ref{sec:sl2}, each nonzero summand of $Z_{d,\lambda}^{A}$ determines, for every $\alpha\in\Delta_{\lambda,+}$, a flow $f_{\alpha}\in\mathcal{F}(K_{D^{\ast}})$.
For every $\alpha \in \Delta_{+}$, put
\[
    t_{\alpha}:=e^{-2(\lambda,\alpha)h}.
\]
Proposition~\ref{prop:semisimple-flow-evaluation} gives
\[
    Z_{d,\lambda}^{A}
    = \prod_{\alpha\in\Delta_{\lambda,+}} \zeta(K_{D^{\ast}},w_{t_{\alpha}}) + O(d^{-1})
\]
coefficientwise.
Finally, Proposition~\ref{prop:semisimple-cap-cup} accounts for the scalar contributions from caps and cups. For $\alpha\in\Delta_{+}\setminus\Delta_{\lambda,+}$, we have $t_{\alpha}=1$ and hence $\zeta(K_{D^{\ast}},w_{t_{\alpha}})=1$. Together with \eqref{eq:alexander-zeta}, this yields Theorem~\ref{thm:mmr-semisimple}:
\[
    \lim_{d\to\infty}
    \left.
    \frac{J_{\mathfrak{g}}(K,V_{d\lambda})}
            {\operatorname{dim}_{q}(V_{d\lambda})}
    \right|_{q=e^{h/d}}
    =
    \prod_{\alpha\in\Delta_{+}}
    \Delta_{K}\!\left(e^{-2(\lambda,\alpha)h}\right)^{-1}.
\]
The limit is coefficientwise in $h$.

\section*{Acknowledgements}
I would like to express my sincere gratitude to Takuya Sakasai, my doctoral supervisor, for his invaluable guidance and encouragement. I also thank Hiroshi Goda, Takayuki Morifuji, and Eri Hatakenaka for their helpful comments. Finally, I am deeply grateful to my wife for her constant support, encouragement, and patience throughout this work.

\vspace{\baselineskip}

\underline{\textbf{Use of AI.}}
The author used AI tools for proofreading, proof polishing, language polishing, formatting checks, and bibliographic metadata checks. The author assumes full responsibility for all mathematical content, arguments, citations, and conclusions in this manuscript.

% =========================
%  Alexander polynomial and graph zeta functions
% =========================
\section{Alexander polynomials and graph zeta functions}
\label{sec:zeta_alex}

Let $K$ be an oriented knot. Throughout this paper, we fix an oriented $(1,1)$--tangle diagram $D^{\ast}$ whose closure is $K$ and whose writhe is zero. We further assume that both strands at every crossing of $D^{\ast}$ are oriented downward.

Label the arcs of $D^{\ast}$ along the orientation by
\[
    x_{1},x_{2},\ldots,x_{n},
\]
so that, for each $1\leq i\leq n-1$, the arc $x_{i}$ is the incoming under-arc at a crossing and $x_{i+1}$ is the outgoing under-arc. Let $x_{j}$ be the over-arc at this crossing, and let $\epsilon_{i}\in\{1,-1\}$ denote its sign. Then the knot group $\pi_{K}:=\pi_{1}(S^{3}\setminus K)$ has the Wirtinger presentation
\[
    \pi_{K}
    =
    \left\langle
        x_{1},\ldots,x_{n}
        \,\middle|\,
        r_{1},\ldots,r_{n-1},\,x_{1}=x_{n}
    \right\rangle,
\]
where $r_{i} := x_{i} \left(x_{j}^{\epsilon_{i}} x_{i+1} x_{j}^{-\epsilon_{i}} \right)^{-1}$. Let
\[
    \alpha : \mathbb{Z}[\pi_{K}] \longrightarrow \mathbb{Z}[t,t^{-1}]
\]
be the homomorphism induced by the abelianization
\[
    \pi_{K} \longrightarrow H_{1}(S^{3}\setminus K) \cong \mathbb{Z}, \qquad
    x_{i} \longmapsto t.
\]
Put $r_{n} := x_{1}x_{n}^{-1}$, and let
\[
    F_{D^{\ast}}
    :=
    \left(
        \alpha\!\left(\frac{\partial r_{i}}{\partial x_{j}}\right)
    \right)_{1 \leq i,j \leq n}
\]
be the corresponding Alexander matrix. Let $M_{n}$ be the matrix obtained from $F_{D^{\ast}}$ by deleting the $n$-th row and the $n$-th column. We set
\[
    \Delta_{D^{\ast}}(t) := \det M_{n}.
\]

We next associate a weighted directed graph to $D^{\ast}$. Let $K_{D^{\ast}}$ be the directed graph with vertex set
\[
    V(K_{D^{\ast}}) := \{v_{1},v_{2},\ldots,v_{n}\}.
\]
For $1 \leq i \leq n-1$, the vertex $v_{i}$ corresponds to the crossing at which $x_{i}$ is the incoming under-arc, and $v_{n}$ corresponds to the arc $x_{n}$ containing the terminal endpoint of $D^{\ast}$. At the crossing corresponding to $v_{i}$, let $x_{j}$ be the over-arc. We put two directed edges
\[
    e_{i}^{0} : v_{i} \longrightarrow v_{i+1}, \qquad
    e_{i}^{1} : v_{i} \longrightarrow v_{j},
\]
and define their weights by
\[
    w_{t}(e_{i}^{0}) = t^{\epsilon_{i}}, \qquad
    w_{t}(e_{i}^{1}) = 1 - t^{\epsilon_{i}}.
\]
If $e_{i}^{0}$ and $e_{i}^{1}$ have the same initial and terminal vertices, we still regard them as distinct edges.
For a positive crossing, the correspondence is illustrated as follows:
\begin{center}
    \begin{tikzpicture}[x=1.2cm,y=1.2cm]
        \node (xi1) at (1,-1) {};
        \node (xi) at (-1,1) {};
        \node (xj1) at (-1,-1) {};
        \node (xj) at (1,1) {};

        \draw (xi1) node [right] {$x_{i+1}$};
        \draw (xi) node [left] {$x_{i}$};
        \draw (xj1) node [left] {$x_{j}$};

        \draw[line width=1pt, ->] (xi) -> (xi1);
        \draw[color=white, line width=5pt] (xj) -- (xj1);
        \draw[line width=1pt, ->] (xj) -> (xj1);

        \draw[arrows={triangle 90-triangle 90}] (2.2,0) -- (3.2,0);

        \node (A0) at (6,-1) {};
        \node (A1) at (5,1) {};
        \node (A2) at (7,1) {};

        \draw (A0) node [right] {$v_{i}$};
        \draw (A1) node [above] {$v_{i+1}$};
        \draw (A2) node [above] {$v_{j}$};

        \fill (A0) circle [radius=1pt];
        \fill (A1) circle [radius=1pt];
        \fill (A2) circle [radius=1pt];

        \draw[line width=1pt, ->] (A0) to node[left] {$t$} (A1);
        \draw[line width=1pt, ->] (A0) to node[right] {$1-t$} (A2);
    \end{tikzpicture}
\end{center}

Let $B_{D^{\ast}}$ be the adjacency matrix of $K_{D^{\ast}}$, that is, the $(a,b)$-entry of $B_{D^{\ast}}$ is the sum of the weights of all edges from $v_{a}$ to $v_{b}$. A direct computation of the Fox derivatives shows that
\[
    \Delta_{D^{\ast}}(t) = \det(I - B_{D^{\ast}}).
\]

\begin{dfn}[Graph zeta function]
    Let $C=(e_{1},\ldots,e_{r})$ be a directed cycle in $K_{D^{\ast}}$. We set
    \[
        w_{t}(C) := \prod_{k=1}^{r} w_{t}(e_{k}).
    \]
    Two directed cycles are identified if they differ by a cyclic permutation. A directed cycle is called prime if it is not a proper power of a shorter directed cycle. The graph zeta function of $(K_{D^{\ast}},w_{t})$ is defined by
    \[
        \zeta(K_{D^{\ast}},w_{t}) := \prod_{[C]} \left(1-w_{t}(C)\right)^{-1},
    \]
    where $[C]$ runs over the equivalence classes of prime directed cycles.
\end{dfn}

By the determinant formula for graph zeta functions, together with
\[
    \Delta_{D^{\ast}}(t) = \det(I-B_{D^{\ast}}),
\]
we have
\[
    \zeta(K_{D^{\ast}},w_{t}) = \det(I-B_{D^{\ast}})^{-1} = \Delta_{D^{\ast}}(t)^{-1};
\]
see, for example, \cite{lin2001random}.
We next recall the flow expansion of the graph zeta function.

\begin{dfn}[{\cite{garoufalidis2006non}}]
    A flow $f$ on $K_{D^{\ast}}$ is a function
    \[
        f : E(K_{D^{\ast}}) \longrightarrow \mathbb{Z}_{\geq 0}
    \]
    satisfying the conservation law
    \[
        \sum_{e:\,e\text{ starts at }v} f(e)
        = \sum_{e:\,e\text{ ends at }v} f(e)
    \]
    at every vertex $v$ of $K_{D^{\ast}}$. We denote this common value by $f(v)$, and write $\mathcal{F}(K_{D^{\ast}})$ for the set of flows on $K_{D^{\ast}}$.
    For a flow $f$, its weight is
    \[
        w_{t}(f)
        := \prod_{e\in E(K_{D^{\ast}})} w_{t}(e)^{f(e)}.
    \]
    If the outgoing edges from $v$ are $e_{1},\ldots,e_{s}$, the multiplicity at a vertex $v$ is
    \[
        \operatorname{mult}_{v}(f)
        := \frac{f(v)!} {f(e_{1})!\cdots f(e_{s})!}.
    \]
    The multiplicity of $f$ is
    \[
        \operatorname{mult}(f)
        := \prod_{v\in V(K_{D^{\ast}})} \operatorname{mult}_{v}(f).
    \]
\end{dfn}

\begin{lem}[{\cite[Theorem~3]{garoufalidis2006non}}]\label{lem:zeta_flow}
    The graph zeta function has the flow expansion
    \[
        \zeta(K_{D^{\ast}},w_{t})
        = \sum_{f\in\mathcal{F}(K_{D^{\ast}})} \operatorname{mult}(f)w_{t}(f).
    \]
\end{lem}

Let $\operatorname{rot}(D^{\ast})$ denote the rotation number of $D^{\ast}$, defined as the sum of the rotation numbers of the oriented circles obtained by smoothing all crossings.
\[
    \Delta_{K}(t)
    := t^{\frac{\operatorname{rot}(D^{\ast})}{2}} \Delta_{D^{\ast}}(t).
\]
This normalization agrees with the one used by Lin and Wang \cite{lin2001random}. Equivalently,
\begin{equation} \label{eq:alexander-zeta}
    t^{-\frac{\operatorname{rot}(D^{\ast})}{2}}
    \zeta(K_{D^{\ast}},w_{t}) = \Delta_{K}(t)^{-1}.
\end{equation}

At $t=1$, every $e_{i}^{1}:v_{i}\to v_{j}$ has weight $0$ and every $e_{i}^{0}:v_{i}\to v_{i+1}$ has weight $1$. Since no directed cycle consists only of edges of the form $e_{i}^{0}$, every directed cycle has weight $0$. Hence
\[
    \zeta(K_{D^{\ast}},w_{1}) = \Delta_{K}(1) = 1.
\]

%%%%%%%%%%%%%%%%%%%%%%%%%%%%%%%%%%%%%%%%%%%%%%%%%%%%%%%%%%%
%%\S3
%%%%%%%%%%%%%%%%%%%%%%%%%%%%%%%%%%%%%%%%%%%%%%%%%%%%%%%%%%%%%%

\section{The MMR conjecture for \texorpdfstring{$\mathcal{U}_{q}(\mathfrak{sl}_{2})$}{Uq(sl2)}}
\label{sec:sl2}

\subsection{Preliminaries and setup}\label{subsec:sl2_setup}

We use the Reshetikhin--Turaev functor \cite{ReshetikhinTuraev1990} for framed oriented tangles coloured by finite-dimensional type-one $\mathcal{U}_{q}(\mathfrak{sl}_{2})$--modules, and write $\mathcal{F}(T;V)$ for the operator assigned to a tangle $T$ coloured by $V$.

For $m\geq 0$, let $V_{m+1}$ denote the $(m+1)$--dimensional irreducible $\mathcal{U}_{q}(\mathfrak{sl}_{2})$--module of highest weight $m$. We use a weight basis $v_{0},\ldots,v_{m}$ with
\[
    H v_{k} = (m-2k)v_{k}, \qquad
    E v_{k} = [m-k+1]_{q}v_{k-1}, \qquad
    F v_{k} = [k+1]_{q}v_{k+1},
\]
where $v_{-1}=v_{m+1}=0$ and
\[
    [n]_{q} := \frac{q^{n}-q^{-n}}{q-q^{-1}}, \qquad
    [n]_{q}! := [1]_{q}\cdots[n]_{q}, \qquad
    [0]_{q}!:=1.
\]
For $V_{2}$, we write
\[
    x_{+}:=v_{0}, \qquad
    x_{-}:=v_{1}.
\]
In particular,
\[
    Hx_{+}=x_{+}, \qquad
    Hx_{-}=-x_{-}, \qquad
    Ex_{-}=x_{+}, \qquad
    Fx_{+}=x_{-},
\]
and $E x_{+} = F x_{-} = 0$.
With these conventions, the universal $\mathcal{R}$--matrix is
\[
    \mathcal{R}
    = q^{\frac{1}{2}H\otimes H}
    \sum_{n\geq 0} q^{\frac{1}{2}n(n-1)} 
    \frac{(q-q^{-1})^{n}}{[n]_{q}!} E^{n}\otimes F^{n}.
\]

Equip the fixed diagram $D^{\ast}$ of \S\ref{sec:zeta_alex} with the blackboard framing, and let $T$ be the framed $(1,1)$--tangle represented by $D^{\ast}$. Since the writhe of $D^{\ast}$ is zero, its closure represents $K$ with framing zero. Put
\[
    F_{T}^{(d)} := \mathcal{F}(T;V_{d+1}) \in \operatorname{End}(V_{d+1}), \qquad
    \widetilde{F}_{T}^{(d)} := \mathcal{F}(T;V_{2}^{\otimes d}) \in \operatorname{End}(V_{2}^{\otimes d}).
\]

The weight space of weight $d$ of $V_{2}^{\otimes d}$ is one-dimensional, spanned by $x_{+}^{\otimes d}$. Hence $V_{d+1}$ occurs in $V_{2}^{\otimes d}$ with multiplicity one. We therefore fix the $\mathcal{U}_{q}(\mathfrak{sl}_{2})$--module embedding $\iota_{d} : V_{d+1} \hookrightarrow V_{2}^{\otimes d}$ by requiring
\[
    \iota_{d}(v_{0}) = x_{+}^{\otimes d},
\]
where $v_{0}$ is the highest-weight vector of $V_{d+1}$.
By Schur's lemma, $F_{T}^{(d)}$ is a scalar operator. We denote this scalar by $Z_{d}^{RT}$:
\[
    F_{T}^{(d)} = Z_{d}^{RT}\, \operatorname{id}_{V_{d+1}}.
\]
With our normalization,
\[
    \frac{J_{\mathfrak{sl}_2}(K,V_{d+1})}{[d+1]_{q}}
    = \frac{1}{[d+1]_{q}} \operatorname{Tr}_{V_{d+1}} \left( q^{H}F_{T}^{(d)} \right).
\]
Since $\operatorname{Tr}_{V_{d+1}}(q^{H}) = [d+1]_{q}$, it follows that
\[
    Z_{d}^{RT} = \frac{J_{\mathfrak{sl}_2}(K,V_{d+1})}{[d+1]_{q}}.
\]
By naturality with respect to the module morphism $\iota_{d}$,
\[
    \widetilde{F}_{T}^{(d)} \circ \iota_{d}
    = \iota_{d} \circ F_{T}^{(d)}.
\]
Applying this identity to $v_{0}$ gives
\[
    \widetilde{F}_{T}^{(d)} x_{+}^{\otimes d}
    = Z_{d}^{RT} x_{+}^{\otimes d}.
\]
We now write $Z_{d}^{RT}$ as a state sum over the fixed diagram $D^{\ast}$.
Put
\[
    I_{d}:=\{0,1,\ldots,d-1\},
\]
and let
\[
    \mathcal{B}_{d}
    :=
    \left\{
        x_{\epsilon_{0}}\otimes\cdots\otimes x_{\epsilon_{d-1}}
        \,\middle|\,
        \epsilon_{a}\in\{+,-\}
    \right\}
\]
be the basis of $V_{2}^{\otimes d}$. We index the tensor factors of $V_{2}^{\otimes d}$ by $a\in I_{d}$; thus, for
\[
    b=x_{\epsilon_{0}}\otimes\cdots\otimes x_{\epsilon_{d-1}}
    \in\mathcal{B}_{d},
\]
the vector in the tensor factor indexed by $a$ is $x_{\epsilon_{a}}$.
For $b\in\mathcal{B}_{d}$, let $b^{\vee}\in\left(V_{2}^{\otimes d}\right)^{*}$ denote the corresponding dual basis vector. 
Thus $b$ and $b^{\vee}$ carry the same label, but they are not identified as vectors. We regard $\left(V_{2}^{\otimes d}\right)^{*}$ as the dual of the whole tensor product and do not use an identification with $\left(V_{2}^{*}\right)^{\otimes d}$.
When $b\in\mathcal{B}_{d}$ is assigned to an oriented segment of the diagram, it represents $b$ on a downward-oriented segment and $b^{\vee}$ on an upward-oriented segment. By the convention fixed in \S\ref{sec:zeta_alex}, both strands at every crossing are oriented downward. Therefore the coefficient at a crossing is the corresponding matrix coefficient of $\mathcal{R}^{\pm1}$ on
\[
    V_{2}^{\otimes d}\otimes V_{2}^{\otimes d},
\]
and no dual module occurs at a crossing.
At caps and cups, the local maps are the evaluation and coevaluation morphisms prescribed by the orientations in the Reshetikhin--Turaev construction \cite{ReshetikhinTuraev1990}. With respect to $\mathcal{B}_{d}$ and its dual basis, the coefficient involving $b\in\mathcal{B}_{d}$ and $c^{\vee}$ is nonzero only when $b=c$; when $b=c$, it is a scalar determined by the type of cap or cup and the weight of $b$. Hence the two segments adjacent to a cap or cup must carry the same label $b\in\mathcal{B}_{d}$, and the cap or cup contributes this scalar.

From the fixed diagram $D^{\ast}$ of $T$, form an oriented planar graph $\Gamma_{T}$ by regarding each crossing as a vertex and the two boundary points as boundary vertices. The edges of $\Gamma_{T}$ are the oriented portions of the diagram between these vertices. Thus every crossing vertex has two incoming and two outgoing edges.

A state is a map
\[
    \sigma : E(\Gamma_{T})\longrightarrow\mathcal{B}_{d}
\]
such that $\sigma$ takes the value $x_{+}^{\otimes d}$ on the two boundary edges and, at every crossing, the corresponding matrix coefficient of $\mathcal{R}^{\pm1}$ determined by the values of $\sigma$ on the four incident edges is nonzero. Let $S_{d}$ denote the set of states.

For $\sigma\in S_{d}$, let $W_{d}^{RT}(\sigma)$ be the product of the corresponding matrix coefficients at all crossings and the scalar contributions from all caps and cups. Then
\[
    Z_{d}^{RT}
    = \sum_{\sigma\in S_{d}} W_{d}^{RT}(\sigma).
\]
From now on, we put
\[
    q = e^{h/d}, \qquad
    t := q^{-2d} = e^{-2h}.
\]

We next describe the operator assigned to a crossing on $V_{2}^{\otimes d}\otimes V_{2}^{\otimes d}$.
Write
\[
    V_{2}^{\otimes d}\otimes V_{2}^{\otimes d}
    =
    \left(
        V_{2}^{(u,0)}
        \otimes\cdots\otimes
        V_{2}^{(u,d-1)}
    \right)
    \otimes
    \left(
        V_{2}^{(o,0)}
        \otimes\cdots\otimes
        V_{2}^{(o,d-1)}
    \right),
\]
where $V_{2}^{(u,i)}$ and $V_{2}^{(o,j)}$ denote the tensor factors with indices $i$ and $j$ on the under-strand and the over-strand, respectively.
For $(i,j)\in I_{d}\times I_{d}$, let $r_{u_{i};o_{j}}^{+}$ and $r_{u_{i};o_{j}}^{-}$ denote the operators on $V_{2}^{\otimes d}\otimes V_{2}^{\otimes d}$ that act as $\mathcal{R}$ and $\mathcal{R}^{-1}$, respectively, on
\[
    V_{2}^{(u,i)}\otimes V_{2}^{(o,j)}
\]
and as the identity on all other tensor factors.

By the quasitriangular identities for $\mathcal{R}$ and coassociativity of the coproduct, the operator assigned to a positive crossing is an ordered product of the operators $r_{u_{i};o_{j}}^{+}$; see \cite{ReshetikhinTuraev1990}. With the tensor factors ordered as above, these operators act in the following order:
\[
    \begin{aligned}
        &(d-1,0),(d-1,1),\ldots,(d-1,d-1),\\
        &(d-2,0),(d-2,1),\ldots,(d-2,d-1),\\
        &\hspace{2.5cm}\vdots\\
        &(0,0),(0,1),\ldots,(0,d-1).
    \end{aligned}
\]
Here and below, the pairs $(i,j)$ are listed in the order in which the corresponding operators act. For a negative crossing, the operators $r_{u_{i};o_{j}}^{-}$ act in the reverse of the order displayed above.
We call each occurrence $r_{u_{i};o_{j}}^{\pm}$ in such an ordered product an \emph{elementary factor}.

We also use the following notation for coefficientwise estimates. For $f\in\mathbb{C}[[h]]$ and $m\geq 0$, define
\[
    |f|_{(m)}
    := \sum_{\ell=0}^{m} \left|[h^{\ell}]f\right|.
\]
For a $d$--dependent family $f_{d}\in\mathbb{C}[[h]]$ and $s\in\mathbb{R}$, we write
\[
    f_{d}=O(d^{s})\quad \text{coefficientwise}
    :\Longleftrightarrow 
    [h^{\ell}]f_{d}=O(d^{s})\quad \text{ for every fixed }\ell
    \geq 0 \text{ as } d\to\infty.
\]
For two such families $f_{d}$ and $g_{d}$, we write
\[
    f_{d}=g_{d}+O(d^{s})\quad \text{coefficientwise}
    :\Longleftrightarrow f_{d}-g_{d}=O(d^{s})\quad \text{coefficientwise}.
\]

%%%%%%%%%%%%%%%%%%%%%%%%%%%%%%%%%%%%%%%%%%%%%%%%%%%%%%%%%%%%%%%%%%%%%%%%%%%%%%%%
%%%%%%%%%%%%%%%%%%%%%%%%%%%%%%%%%%%%%%%%%%%%%%%%%%%%%%%%%%%%%%%%%%%%%%%%%%%%%%%%
%%%%%%%%%%%%%%%%%%%%%%%%%%%%%%%%%%%%%%%%%%%%%%%%%%%%%%%%%%%%%%%%%%%%%%%%%%%%%%%%

\subsection{Reduction to the leading state sum}

For $\sigma\in S_{d}$, let $W_{d}^{\mathrm{cross}}(\sigma)$ and $W_{d}^{\mathrm{cap/cup}}(\sigma)$ denote the products of the coefficients contributed by the crossings and by the caps and cups, respectively. Then
\[
    W_{d}^{RT}(\sigma)
    = W_{d}^{\mathrm{cap/cup}}(\sigma) W_{d}^{\mathrm{cross}}(\sigma).
\]
Let
\[
    \mathbf{1} := x_{+}\otimes x_{+}, \qquad
    l_{u} := x_{-}\otimes x_{+}, \qquad
    l_{o} := x_{+}\otimes x_{-}, \qquad
    l_{uo} := x_{-}\otimes x_{-}.
\]

\begin{lem} \label{lem:V2Coeff}
    Put
    \[
        \bar r^{\pm}
        := q^{\mp 1/2} \mathcal{R}^{\pm 1} \big|_{V_{2}\otimes V_{2}}.
    \]
    Then
    \[
        \bar r^{\pm}(\mathbf{1})=\mathbf{1},\quad
        \bar r^{\pm}(l_{u})=q^{\mp 1}l_{u}+(1-q^{\mp 2})l_{o},\quad
        \bar r^{\pm}(l_{o})=q^{\mp 1}l_{o},\quad
        \bar r^{\pm}(l_{uo})=l_{uo}.
    \]
\end{lem}

\begin{proof}
    This follows by a direct calculation from the formula for $\mathcal{R}$ on $V_{2}\otimes V_{2}$.
\end{proof}

Replacing each elementary factor $r_{u_{i};o_{j}}^{\pm}$ by the corresponding copy of $\bar r^{\pm}$ multiplies the coefficient at a positive crossing by $q^{-d^{2}/2}$ and that at a negative crossing by $q^{d^{2}/2}$. Writing $\operatorname{wr}(D^{\ast})$ for the writhe of $D^{\ast}$, the product of these scalars over all crossings is
\[
    q^{-\frac{d^{2}}{2}\operatorname{wr}(D^{\ast})}.
\]
Since $\operatorname{wr}(D^{\ast})=0$, this product is equal to $1$. Therefore, for every $\sigma\in S_{d}$, $W_{d}^{\mathrm{cross}}(\sigma)$ is unchanged if $r_{u_{i};o_{j}}^{\pm}$ is replaced by the corresponding copy of $\bar r^{\pm}$ at every elementary factor.

Define $r_{A}^{+}:V_{2}\otimes V_{2}\to V_{2}\otimes V_{2}$ by
\[
    r_{A}^{+}(\mathbf{1}) = \mathbf{1},\quad
    r_{A}^{+}(l_{u}) = q^{- 2}l_{u}+(1-q^{- 2})l_{o},\quad
    r_{A}^{+}(l_{o}) = l_{o},\quad
    r_{A}^{+}(l_{uo}) = q^{-2}l_{uo},
\]
and define $r_{A}^{-}$ by replacing $q$ with $q^{-1}$ in these formulas.

Following Lin and Wang, we call the choice of the $l_{o}$--term in $\bar r^{\pm}(l_{u})$, and likewise in $r_{A}^{\pm}(l_{u})$, a \emph{jump}; this corresponds to the jump-up move in their random-walk description \cite{lin2001random}.
For each $v\in\{\mathbf{1},l_{u},l_{o},l_{uo}\}$, the vectors occurring with nonzero coefficients in $\bar r^{\pm}(v)$ and $r_{A}^{\pm}(v)$ are the same.
Fix $\sigma\in S_{d}$. Using the ordered products of elementary factors, we write $W_{d}^{\mathrm{cross}}(\sigma)$ as a finite sum. We index the resulting summands by a finite set $\mathcal{E}_{d}(\sigma)$, where $\eta\in\mathcal{E}_{d}(\sigma)$ records the choices of nonzero summands made as the elementary factors are successively applied to produce the corresponding summand.

For $\eta\in\mathcal{E}_{d}(\sigma)$, let $W_{d}^{\mathrm{cross}}(\eta)$ be the product of the coefficients of $\bar r^{\pm}$ selected by $\eta$. Since the same choices are available for $r_{A}^{\pm}$, let $W_{d}^{\mathrm{lead}}(\eta)$ be the product of the corresponding coefficients of $r_{A}^{\pm}$. Then
\[
    W_{d}^{\mathrm{cross}}(\sigma)
    = \sum_{\eta\in\mathcal{E}_{d}(\sigma)} W_{d}^{\mathrm{cross}}(\eta),
\]
and we define
\[
    W_{d}^{\mathrm{lead}}(\sigma)
    := \sum_{\eta\in\mathcal{E}_{d}(\sigma)} W_{d}^{\mathrm{lead}}(\eta).
\]

Let $r(\eta)$ denote the number of jumps recorded by $\eta$.

\begin{rem}\label{rem:coefficientwise-finiteness}

    Fix $\sigma\in S_{d}$ and $\eta\in\mathcal{E}_{d}(\sigma)$. The choices recorded by $\eta$ determine oriented curves by continuing each occurrence of $x_{-}$ through the elementary factors. Since $\sigma$ takes the value $x_{+}^{\otimes d}$ on both boundary edges, these curves have no boundary endpoints and are therefore closed.
    A curve containing no jump would continue along the oriented tangle and reach a boundary edge. Hence every such curve contains at least one jump.
    It follows that, for every edge $e$ of $\Gamma_{T}$, $\sigma(e)$ contains at most $r(\eta)$ occurrences of $x_{-}$.

    Under $q=e^{h/d}$, every coefficient corresponding to a jump satisfies
    \[
        1-q^{\pm2} \in h\mathbb{C}[d^{-1}][[h]].
    \]
    Hence both $W_{d}^{\mathrm{cross}}(\eta)$ and $W_{d}^{\mathrm{lead}}(\eta)$ are divisible by $h^{r(\eta)}$. Consequently, when considering coefficients of degree at most $m$, only $\eta$ with $r(\eta)\leq m$ can contribute. For every such $\eta$ and every edge $e$ of $\Gamma_{T}$, $\sigma(e)$ contains at most $m$ occurrences of $x_{-}$, independently of $d$.
\end{rem}

\begin{prop}\label{prop:kappa-bound}
    For every $\sigma\in S_{d}$ and every $\eta\in\mathcal{E}_{d}(\sigma)$, there exists an integer $\kappa(\eta)$ such that
    \[
        W_{d}^{\mathrm{lead}}(\eta)
        = q^{\kappa(\eta)}W_{d}^{\mathrm{cross}}(\eta).
    \]
    Moreover, for every $m\geq 0$, there exists a constant $C_{T,m}$, independent of $d$, $\sigma$, and $\eta$, such that if $r(\eta)\leq m$, then
    \[
        |\kappa(\eta)|\leq C_{T,m}.
    \]
    Consequently,
    \[
        \left|q^{\kappa(\eta)}-1\right|_{(m)}
        = O(d^{-1})
    \]
    uniformly over all $\sigma\in S_{d}$ and $\eta\in\mathcal{E}_{d}(\sigma)$ satisfying $r(\eta)\leq m$.
\end{prop}

\begin{proof}
    Fix $\sigma\in S_{d}$ and $\eta\in\mathcal{E}_{d}(\sigma)$. After joining the two boundary points outside the diagram, the occurrences of $x_{-}$ and $x_{+}$, traced according to the choices recorded by $\eta$, form oriented closed curves.

    We first compare the coefficients of $l_{u}\mapsto l_{u}$ and $l_{o}\mapsto l_{o}$.
    For a positive crossing, they change as
    \[
        l_{u}\mapsto l_{u}: \quad
        q^{-1}\rightsquigarrow q^{-2}, \qquad
        l_{o}\mapsto l_{o}: \quad
        q^{-1}\rightsquigarrow 1,
    \]
    and for a negative crossing the corresponding formulas are obtained by replacing $q$ with $q^{-1}$.

    \begin{center}
        \begin{tikzpicture}[scale=0.8,>=Latex,line cap=round,line join=round]
            % left picture: x_+ over x_-
            \begin{scope}[xshift=0cm]
                \draw[->,line width=0.9pt] (-1.8,0) -- (1.8,0);
                \draw[line width=0.9pt] (0,-1.5) -- (0,-0.28);
                \draw[->,line width=0.9pt] (0,0.28) -- (0,1.5);

                \node[above] at (0,1.5) {$x_{-}$};
                \node[right] at (1.8,0) {$x_{+}$};
                \node[below] at (0,-1.9) {$q^{-1}\rightsquigarrow q^{-2}$};
            \end{scope}

            % right picture: x_- over x_+
            \begin{scope}[xshift=5.3cm]
                \draw[->,line width=0.9pt] (0,-1.5) -- (0,1.5);
                \draw[->,line width=0.9pt] (-1.8,0) -- (-0.28,0);
                \draw[->,line width=0.9pt] (0.28,0) -- (1.8,0);

                \node[above] at (0,1.5) {$x_{-}$};
                \node[right] at (1.8,0) {$x_{+}$};
                \node[below] at (0,-1.9) {$q\rightsquigarrow 1$};
            \end{scope}
        \end{tikzpicture}
    \end{center}

    Let $\gamma_{-}$ be one of the closed curves obtained by tracing $x_{-}$ and let $\gamma_{+}$ be one obtained by tracing $x_{+}$. At a transverse intersection where $\gamma_{+}$ crosses $\gamma_{-}$ from left to right, the ratio of the coefficient for $r_{A}^{\pm}$ to the corresponding coefficient for $\bar r^{\pm}$ is $q^{-1}$; at an intersection from right to left, the ratio is $q$. These ratios are independent of the sign of the crossing. The coefficient of $l_{u}\mapsto l_{o}$ is the same for $\bar r^{\pm}$ and $r_{A}^{\pm}$, so a small perturbation may be made near the elementary factors at which jumps are chosen in order to avoid intersections there.
    It follows that the product of these ratios for the intersections of $\gamma_{-}$ and $\gamma_{+}$ is
    \[
        q^{-I(\gamma_{-},\gamma_{+})},
    \]
    where $I(\gamma_{-},\gamma_{+})$ denotes their algebraic intersection number in $\mathbb{R}^{2}$. Since $\gamma_{-}$ and $\gamma_{+}$ are closed curves in $\mathbb{R}^{2}$,
    \[
        I(\gamma_{-},\gamma_{+})=0.
    \]
    Applying this to every pair of such curves, the ratios arising from $l_{u}\mapsto l_{u}$ and $l_{o}\mapsto l_{o}$ multiply to $1$.

    It remains to compare the coefficient of $l_{uo}\mapsto l_{uo}$. This coefficient is $1$ for $\bar r^{\pm}$ and $q^{\mp2}$ for $r_{A}^{\pm}$. Hence the product of the remaining ratios is $q^{\kappa(\eta)}$ for some $\kappa(\eta)\in\mathbb{Z}$, and
    \[
        W_{d}^{\mathrm{lead}}(\eta)
        = q^{\kappa(\eta)}W_{d}^{\mathrm{cross}}(\eta).
    \]

    Suppose now that $r(\eta)\leq m$. By Remark~\ref{rem:coefficientwise-finiteness}, $\sigma(e)$ contains at most $m$ occurrences of $x_{-}$ for every edge $e$ of $\Gamma_{T}$. For each crossing, $l_{uo}\mapsto l_{uo}$ can occur at most $m^{2}$ elementary factors.
    Since each such occurrence contributes $\pm2$ to $\kappa(\eta)$ and the number of crossings of $D^{\ast}$ is fixed, there exists a constant $C_{T,m}$, independent of $d$, $\sigma$, and $\eta$, such that
    \[
        |\kappa(\eta)| \leq C_{T,m}.
    \]
    Finally, since $q=e^{h/d}$,
    \[
        \left| q^{\kappa(\eta)}-1 \right|_{(m)}
        = \sum_{\ell=1}^{m} \frac{|\kappa(\eta)|^{\ell}}{\ell!\,d^{\ell}}
        = O(d^{-1}),
    \]
    uniformly over all $\sigma\in S_{d}$ and $\eta\in\mathcal{E}_{d}(\sigma)$ satisfying $r(\eta)\leq m$.
\end{proof}

We now determine the contribution to $W_{d}^{\mathrm{lead}}(\eta)$ from the $d$ elementary factors $r_{u_{i};o_{j}}^{\pm}$, $j\in I_{d}$, with the same $i$. If the vector in $V_{2}^{(u,i)}$ before these elementary factors is $x_{+}$, the coefficient selected by $\eta$ at each of them is $1$. Hence these elementary factors contribute $1$ to $W_{d}^{\mathrm{lead}}(\eta)$. We therefore consider the case in which this vector is $x_{-}$.
Suppose first that a jump is selected at $r_{u_{i};o_{a}}^{\pm}$, where $a\in I_{d}$. For a positive crossing, the elementary factors
\[
    r_{u_{i};o_{0}}^{+}, \ldots, r_{u_{i};o_{a-1}}^{+}
\]
precede $r_{u_{i};o_{a}}^{+}$, and the coefficient selected at each of them is $q^{-2}$. The coefficient of the jump is $1-q^{-2}$, and the coefficient selected at each of the remaining elementary factors with the same $i$ is $1$. Thus these $d$ elementary factors contribute
\[
    (q^{-2})^{a}(1-q^{-2})
\]
to $W_{d}^{\mathrm{lead}}(\eta)$.
For a negative crossing, the order is reversed, and the corresponding contribution is
\[
    (q^{2})^{d-1-a}(1-q^{2}).
\]
These contributions depend on $a$ but not on $i$. We therefore put
\[
    \psi_{1}^{+}(a;d) := (q^{-2})^{a}(1-q^{-2}), \qquad
    \psi_{1}^{-}(a;d) := (q^{2})^{d-1-a}(1-q^{2}),
\]
for $a\in I_{d}$.

If no jump is selected among the $d$ elementary factors with this $i$, their contribution to $W_{d}^{\mathrm{lead}}(\eta)$ is $(q^{-2})^{d}$ at a positive crossing and $(q^{2})^{d}$ at a negative crossing. We therefore put
\[
    \psi_{0}^{+} := (q^{-2})^{d} = t, \qquad
    \psi_{0}^{-} := (q^{2})^{d} = t^{-1}.
\]

We now compute the contribution of caps and cups. We denote a clockwise cap and cup by $\lcap$ and $\lcup$, respectively, and a counterclockwise cap and cup by $\rcap$ and $\rcup$, respectively. We write $\#\lcap$ for the number of clockwise caps in the diagram, and use the analogous notation for the other three types.
With our convention for the four types of caps and cups, the rotation number defined in \S\ref{sec:zeta_alex} is
\[
    \operatorname{rot}(D^{\ast})
    = \frac{1}{2} \left( \#\lcap+\#\rcup-\#\rcap-\#\lcup \right).
\]
Since the two boundary tangent directions of $D^{\ast}$ agree, the total numbers of caps and cups are equal:
\[
    \#\lcap+\#\rcap = \#\lcup+\#\rcup.
\]
Consequently,
\begin{equation}\label{eq:rot-cap_cup}
    \#\lcap-\#\lcup = \operatorname{rot}(D^{\ast}).
\end{equation}

\begin{prop}[Cap and cup comparison]\label{prop:cap-cup-comparison}
    For every $m\geq 0$, uniformly for all $\sigma\in S_{d}$ and $\eta\in\mathcal{E}_{d}(\sigma)$ satisfying $r(\eta)\leq m$, one has
    \[
        W_{d}^{\mathrm{cap/cup}}(\sigma)
        = t^{-\frac{\operatorname{rot}(D^{\ast})}{2}} \left(1+O(d^{-1})\right)
    \]
    in coefficients of degree at most $m$.
\end{prop}

\begin{proof}
    For a cap or cup contained in an edge $e$ of $\Gamma_{T}$, let $k$ be the number of occurrences of $x_{-}$ in $\sigma(e)$. According to its type $\lcap, \lcup, \rcap$, or $\rcup$, its coefficient is $q^{d-2k}, q^{-d+2k}, 1, 1$, respectively.
    Moreover,
    \[
        q^{d-2k} = t^{-1/2}q^{-2k}, \qquad
        q^{-d+2k} = t^{1/2}q^{2k}.
    \]

    By Remark~\ref{rem:coefficientwise-finiteness}, $k\leq m$ whenever $r(\eta)\leq m$. Since $q=e^{h/d}$,
    \[
        q^{\pm2k} = 1 + O(d^{-1})
    \]
    uniformly in coefficients of degree at most $m$. Since $D^{\ast}$ has finitely many caps and cups,
    \[
        W_{d}^{\mathrm{cap/cup}}(\sigma)
        = t^{-\frac{\#\lcap-\#\lcup}{2}} \left(1+O(d^{-1})\right)
    \]
    uniformly in coefficients of degree at most $m$. The result follows from \eqref{eq:rot-cap_cup}.
\end{proof}

For $r\geq 0$, put
\[
    \mathcal{E}_{d,r}
    :=
    \left\{
        (\sigma,\eta)
        \,\middle|\,
        \sigma\in S_{d},\ 
        \eta\in\mathcal{E}_{d}(\sigma),\ 
        r(\eta)=r
    \right\}.
\]
For each fixed $r$,
\[
    \#\mathcal{E}_{d,r} = O(d^{r}).
\]
Indeed, each jump is selected at an elementary factor $r_{u_i;o_j}^{\pm}$. Along each closed curve described in Remark~\ref{rem:coefficientwise-finiteness}, the index at which $x_-$ occurs is unchanged between successive jumps. Hence, if a jump is selected at $r_{u_i;o_j}^{\pm}$, the value of $i$ is equal to one of the $r$ values of $j$ occurring at the selected elementary factors. The $r$ values of $j$ have $d^r$ possible choices, and, once they are fixed, each of the $r$ values of $i$ has at most $r$ choices. Since $r$ is fixed, there are therefore $O(d^r)$ possibilities for the pairs $(i,j)$ at the $r$ jumps. Since $D^{\ast}$ has finitely many crossings, the choices of the crossings at which these jumps occur are bounded independently of $d$. Once the elementary factors at which the $r$ jumps are selected are fixed, the choices at all remaining elementary factors are determined.

\begin{prop}\label{prop:crossing-leading}
    One has
    \[
        \sum_{\sigma\in S_{d}}W_{d}^{\mathrm{cross}}(\sigma)
        - \sum_{\sigma\in S_{d}}W_{d}^{\mathrm{lead}}(\sigma)
        = O(d^{-1})
    \]
    coefficientwise.
\end{prop}

\begin{proof}
    Fix $m\geq 0$. By Remark~\ref{rem:coefficientwise-finiteness}, only summands indexed by $\eta$ with $r(\eta)\leq m$ can contribute to coefficients of degree at most $m$.
    Let $0\leq r\leq m$ and $(\sigma,\eta)\in\mathcal{E}_{d,r}$. By the formulas for $\psi_{0}^{\pm}$ and $\psi_{1}^{\pm}$, exactly $r$ factors $\psi_{1}^{\pm}(a;d)$ occur in $W_{d}^{\mathrm{lead}}(\eta)$. Moreover, by Remark~\ref{rem:coefficientwise-finiteness}, the number of factors $\psi_{0}^{\pm}$ is bounded independently of $d$. Hence
    \[
        |W_{d}^{\mathrm{lead}}(\eta)|_{(m)}
        = O(d^{-r})
    \]
    uniformly for $(\sigma,\eta)\in\mathcal{E}_{d,r}$ and $0\leq r\leq m$.
    By Proposition~\ref{prop:kappa-bound},
    \[
        W_{d}^{\mathrm{cross}}(\eta)
        = q^{-\kappa(\eta)}W_{d}^{\mathrm{lead}}(\eta),
    \]
    and
    \[
        |q^{-\kappa(\eta)}-1|_{(m)}
        = O(d^{-1})
    \]
    uniformly for $r(\eta)\leq m$. 
    Therefore
    \begin{align*}
        \left|
            \sum_{\sigma\in S_{d}}W_{d}^{\mathrm{cross}}(\sigma)
            -
            \sum_{\sigma\in S_{d}}W_{d}^{\mathrm{lead}}(\sigma)
        \right|_{(m)}
        &=
        \left|
            \sum_{\sigma\in S_{d}}
            \sum_{\eta\in\mathcal{E}_{d}(\sigma)}
            \left(q^{-\kappa(\eta)}-1\right)
            W_{d}^{\mathrm{lead}}(\eta)
        \right|_{(m)}
        \\
        &\leq
        \sum_{r=0}^{m}
        \sum_{(\sigma,\eta)\in\mathcal{E}_{d,r}}
        |q^{-\kappa(\eta)}-1|_{(m)}
        |W_{d}^{\mathrm{lead}}(\eta)|_{(m)}
        \\
        &=
        \sum_{r=0}^{m}
        O(d^{r})\,O(d^{-1})\,O(d^{-r})
        \\
        &=
        O(d^{-1}).
    \end{align*}
\end{proof}

\begin{cor}\label{cor:rt-lead}
    One has
    \[
        \sum_{\sigma\in S_{d}}W_{d}^{RT}(\sigma)
        = t^{-\frac{\operatorname{rot}(D^{\ast})}{2}}
        \sum_{\sigma\in S_{d}}W_{d}^{\mathrm{lead}}(\sigma)
        + O(d^{-1})
    \]
    coefficientwise.
\end{cor}

\begin{proof}
    Fix $m\geq 0$. By Remark~\ref{rem:coefficientwise-finiteness}, it suffices to consider $0\leq r\leq m$.
    By Proposition~\ref{prop:cap-cup-comparison},
    \[
        \left|
            W_{d}^{\mathrm{cap/cup}}(\sigma)
            - t^{-\frac{\operatorname{rot}(D^{\ast})}{2}}
        \right|_{(m)}
        = O(d^{-1}).
    \]
    By Proposition~\ref{prop:kappa-bound} and the bound for $W_{d}^{\mathrm{lead}}(\eta)$ obtained in the proof of Proposition~\ref{prop:crossing-leading},
    \[
        |W_{d}^{\mathrm{cross}}(\eta)|_{(m)} = O(d^{-r})
    \]
    uniformly for $(\sigma,\eta)\in\mathcal{E}_{d,r}$ and $0\leq r\leq m$.
    Therefore
    \begin{align*}
        \Bigg|
            \sum_{\sigma\in S_{d}}W_{d}^{RT}(\sigma)
            - t^{-\frac{\operatorname{rot}(D^{\ast})}{2}}
            \sum_{\sigma\in S_{d}}W_{d}^{\mathrm{cross}}(\sigma)
        \Bigg|_{(m)}
        &=
        \left|
            \sum_{r=0}^{m}
            \sum_{(\sigma,\eta)\in\mathcal{E}_{d,r}}
            \left(
                W_{d}^{\mathrm{cap/cup}}(\sigma)
                -
                t^{-\frac{\operatorname{rot}(D^{\ast})}{2}}
            \right)
            W_{d}^{\mathrm{cross}}(\eta)
        \right|_{(m)} \\
        &\leq \sum_{r=0}^{m} O(d^{r})\,O(d^{-1})\,O(d^{-r}) \\
        &= O(d^{-1}).
    \end{align*}

    By Proposition~\ref{prop:crossing-leading},
    \[
        \sum_{\sigma\in S_{d}} W_{d}^{\mathrm{cross}}(\sigma)
        = \sum_{\sigma\in S_{d}} W_{d}^{\mathrm{lead}}(\sigma)
        + O(d^{-1})
    \]
    coefficientwise. Since $t^{-\frac{\operatorname{rot}(D^{\ast})}{2}}$ is independent of $d$, it follows that
    \[
        \sum_{\sigma\in S_{d}} W_{d}^{RT}(\sigma)
        = t^{-\frac{\operatorname{rot}(D^{\ast})}{2}}
        \sum_{\sigma\in S_{d}}
        W_{d}^{\mathrm{lead}}(\sigma) + O(d^{-1})
    \]
    coefficientwise.
\end{proof}

Put
\[
    Z_{d}^{\mathrm{lead}}
    := \sum_{\sigma\in S_{d}} W_{d}^{\mathrm{lead}}(\sigma).
\]

%%%%%%%%%%%%%%%%%%%%%%%%%%%%%%%%%%%%%%%%%%%%%%%%%%%%%%%%%%%%%%%%%%%
%%%%%%%%%%%%%%%%%%%%%%%%%%%%%%%%%%%%%%%%%%%%%%%%%%%%%%%%%%%%%%%%%%%
%%%%%%%%%%%%%%%%%%%%%%%%%%%%%%%%%%%%%%%%%%%%%%%%%%%%%%%%%%%%%%%%%%%

\subsection{Evaluation of the leading state sum and the MMR conjecture}\label{subsec:lead-evaluation}

For each state $\sigma \in S_{d}$, we define a function $p_{d}(\sigma)$ on the edges of $K_{D^{\ast}}$ and show that it is a flow. Recall the set of flows $\mathcal{F}(K_{D^{\ast}})$ defined in \S\ref{sec:zeta_alex}.

At the crossing corresponding to a vertex $v_{i}$ of $K_{D^{\ast}}$, let $x_{i}$ be the incoming under-arc, $x_{i+1}$ the outgoing under-arc, and $x_{j}$ the over-arc, as in \S\ref{sec:zeta_alex}. Recall that the two edges starting from $v_{i}$ are
\[
    e_{i}^{0} : v_{i}\longrightarrow v_{i+1}, \qquad
    e_{i}^{1} : v_{i}\longrightarrow v_{j}.
\]

For $\sigma\in S_{d}$, define $p_{d}(\sigma)$ by
\[
    p_{d}(\sigma)(e_{i}^{0})
    :=
    \#\left\{
        \begin{array}{c}
            \text{occurrences of $x_{-}$ in the value of $\sigma$ on the edge of $\Gamma_{T}$}\\
            \text{leaving the crossing along $x_{i+1}$}
        \end{array}
    \right\},
\]
and
\[
    p_{d}(\sigma)(e_{i}^{1})
    :=
    \#\left\{
        \begin{array}{c}
            \text{occurrences of $x_{-}$ in the value of $\sigma$ on the edge of $\Gamma_{T}$}\\
            \text{entering the crossing along $x_{i}$}
        \end{array}
    \right\}
    - p_{d}(\sigma)(e_{i}^{0}).
\]
For any $\eta\in\mathcal{E}_{d}(\sigma)$, $p_{d}(\sigma)(e_{i}^{1})$ is the number of jumps recorded by $\eta$ at this crossing, and is therefore nonnegative.

The formulas for $r_{A}^{\pm}$ preserve the total number of occurrences of $x_{-}$. Hence the nonnegative integers defined above satisfy the conservation law at every crossing vertex of $K_{D^{\ast}}$. At the terminal vertex $v_{n}$, the conservation law follows from the boundary value $x_{+}^{\otimes d}$. Therefore $p_{d}(\sigma)$ is a flow on $K_{D^{\ast}}$.

For every crossing vertex $v_{i}$, the value $p_{d}(\sigma)(v_{i})$ is the number of occurrences of $x_{-}$ in the value of $\sigma$ on the edge of $\Gamma_{T}$ entering the crossing along $x_{i}$, and is therefore at most $d$. Hence we obtain a map
\[
    p_{d}:S_{d}\longrightarrow \mathcal{F}(K_{D^{\ast}})_{\leq d},
\]
where
\[
    \mathcal{F}(K_{D^{\ast}})_{\leq d}
    :=
    \left\{
        f\in\mathcal{F}(K_{D^{\ast}})
        \,\middle|\,
        f(v)\leq d
        \text{ for every }v\in V(K_{D^{\ast}})
    \right\}.
\]
For $f\in\mathcal{F}(K_{D^{\ast}})_{\leq d}$, put
\[
    S_{d}(f) := p_{d}^{-1}(f).
\]
Then
\[
    S_{d} = \bigsqcup_{f\in\mathcal{F}(K_{D^{\ast}})_{\leq d}} S_{d}(f).
\]

For $f\in\mathcal{F}(K_{D^{\ast}})_{\leq d}$, all states $\sigma\in S_{d}(f)$ have the same number of occurrences of $x_{-}$ on every edge of $\Gamma_{T}$, while the indices at which $x_{-}$ occurs may depend on $\sigma$.

Fix $f\in\mathcal{F}(K_{D^{\ast}})_{\leq d}$. For $v=v_{i}\in V(K_{D^{\ast}})\setminus\{v_{n}\}$, write $e_{v}^{0}:=e_{i}^{0}$, $e_{v}^{1}:=e_{i}^{1}$, and let $\epsilon(v)\in\{1,-1\}$ be the sign of the corresponding crossing. Put
\[
    r(f)
    := \sum_{v\in V(K_{D^{\ast}})\setminus\{v_{n}\}} f(e_{v}^{1}), \qquad
    s_{\pm}(f)
    :=
    \sum_{\substack{
        v\in V(K_{D^{\ast}})\setminus\{v_{n}\}\\
        \epsilon(v)=\pm1
    }}
    f(e_{v}^{0}).
\]

For every $\sigma\in S_{d}(f)$ and $\eta\in\mathcal{E}_{d}(\sigma)$, exactly $r(f)$ of the choices recorded by $\eta$ are jumps.
Choose an ordering of the crossing vertices. For each $\eta$, list the elementary factors at which jumps are selected, first according to this ordering of the crossings and then, within each crossing, according to the order in which the elementary factors act. Write the $k$-th elementary factor in this list as
\[
    r_{u_{i_{k}};o_{a_{k}}}^{+} \quad\text{or}\quad
    r_{u_{i_{k}};o_{a_{k}}}^{-},
\]
and put $\epsilon_{k}=1$ or $-1$, respectively.
We use the convention $\psi_{1}^{+1}:=\psi_{1}^{+}$, $\psi_{1}^{-1}:=\psi_{1}^{-}$.
Define
\[
    a(\eta) := (a_{1},\ldots,a_{r(f)}) \in I_{d}^{r(f)}.
\]
Since exactly $f(e_{v}^{1})$ jumps occur at the crossing corresponding to each vertex $v$, the sequence $\epsilon_{1},\ldots,\epsilon_{r(f)}$ depends only on $f$ and the chosen ordering of the crossing vertices.
For $a=(a_{1},\ldots,a_{r(f)})\in I_{d}^{r(f)}$, put
\[
    w_{f}(a)
    := (\psi_{0}^{+})^{s_{+}(f)} (\psi_{0}^{-})^{s_{-}(f)}
    \prod_{k=1}^{r(f)} \psi_{1}^{\epsilon_{k}}(a_{k};d).
\]
By the definitions of $r(f)$, $s_{\pm}(f)$, $\psi_{0}^{\pm}$, and $\psi_{1}^{\pm}$,
\[
    W_{d}^{\mathrm{lead}}(\eta) = w_{f}\bigl(a(\eta)\bigr).
\]
For $a\in I_{d}^{r(f)}$, define
\[
    N_{f}(a)
    :=
    \#\left\{
        (\sigma,\eta)
        \,\middle|\,
        \sigma\in S_{d}(f),\
        \eta\in\mathcal{E}_{d}(\sigma),\
        a(\eta)=a
    \right\}.
\]
Then
\[
    \sum_{\sigma\in S_{d}(f)} W_{d}^{\mathrm{lead}}(\sigma)
    = \sum_{a\in I_{d}^{r(f)}} N_{f}(a)\,w_{f}(a).
\]

For a prescribed $a$, every pair counted by $N_{f}(a)$ is determined by choosing, at each crossing vertex $v$, which $f(e_{v}^{1})$ of the $f(v)$ occurrences of $x_{-}$ correspond to elementary factors at which jumps are selected. Hence
\[
    0 \leq N_{f}(a)
    \leq \prod_{v\in V(K_{D^{\ast}})\setminus\{v_{n}\}} \binom{f(v)}{f(e_{v}^{1})}
    = \operatorname{mult}(f).
\]

Define
\[
    A_{d}(f)
    :=
    \left\{
        a\in I_{d}^{r(f)}
        \,\middle|\,
        N_{f}(a)=\operatorname{mult}(f)
    \right\}.
\]

\begin{lem}\label{lem:bad-set}
    For every $f\in\mathcal{F}(K_{D^{\ast}})$ and every $d$ such that $f\in\mathcal{F}(K_{D^{\ast}})_{\leq d}$, one has
    \[
        \#\left(
            I_{d}^{r(f)}\setminus A_{d}(f)
        \right)
        \leq
        \binom{r(f)}{2}d^{r(f)-1}.
    \]
\end{lem}

\begin{proof}
    Put $r=r(f)$. If $r=0$, the assertion is immediate. Let
    \[
        a=(a_{1},\ldots,a_{r}) \in I_{d}^{r}\setminus A_{d}(f).
    \]
    Since $N_{f}(a)<\operatorname{mult}(f)$, at least one of the $\operatorname{mult}(f)$ choices counted above does not give a pair $(\sigma,\eta)$ with $a(\eta)=a$. Consider such a choice.

    By Remark~\ref{rem:coefficientwise-finiteness}, every occurrence of $x_{-}$ lies on a closed oriented curve containing at least one jump. If the $k$-th jump is selected at an elementary factor $r_{u_{i};o_{a_{k}}}^{\pm}$, the resulting $x_{-}$ occurs at index $a_{k}$. At elementary factors at which no jump is selected, the index of each occurrence of $x_{-}$ is unchanged. Hence, once this choice is fixed, $a$ determines the index of every occurrence of $x_{-}$ on every edge.
    If $a_{i}\neq a_{j}$ for all distinct $i,j\in\{1,\ldots,r\}$, then the fixed choice gives a pair $(\sigma,\eta)$ with $a(\eta)=a$.
    Therefore
    \[
        I_{d}^{r}\setminus A_{d}(f)
        \subset \bigcup_{1\leq i<j\leq r}
        \left\{
            a\in I_{d}^{r}
            \,\middle|\,
            a_{i}=a_{j}
        \right\}.
    \]
    Since each set on the right has cardinality $d^{r-1}$,
    \[
        \#\left(
            I_{d}^{r}\setminus A_{d}(f)
        \right)
        \leq \binom{r}{2}d^{r-1}.
    \]
\end{proof}

\begin{prop}\label{prop:fixed-flow-evaluation}
    For each fixed $f\in\mathcal{F}(K_{D^{\ast}})$, once $d$ is sufficiently large so that $f \in \mathcal{F}(K_{D^{\ast}})_{\leq d}$, one has
    \[
        \sum_{\sigma\in S_{d}(f)} W_{d}^{\mathrm{lead}}(\sigma)
        = \operatorname{mult}(f) t^{s_{+}(f)} (t^{-1})^{s_{-}(f)} \prod_{k=1}^{r(f)} \bigl(1-t^{\epsilon_{k}}\bigr)
        + O(d^{-1})
    \]
    coefficientwise as $d\to\infty$.
\end{prop}

\begin{proof}
    By the preceding decomposition,
    \[
        \sum_{\sigma\in S_{d}(f)} W_{d}^{\mathrm{lead}}(\sigma)
        = \sum_{a\in I_{d}^{r(f)}} N_{f}(a)w_{f}(a).
    \]
    Since $N_{f}(a)=\operatorname{mult}(f)$ for $a\in A_{d}(f)$,
    \[
        \sum_{a\in I_{d}^{r(f)}}
        \bigl( N_{f}(a)-\operatorname{mult}(f) \bigr)w_{f}(a)
        = \sum_{a\in I_{d}^{r(f)}\setminus A_{d}(f)}
        \bigl( N_{f}(a)-\operatorname{mult}(f) \bigr)w_{f}(a).
    \]
    By Lemma~\ref{lem:bad-set},
    \[
        \#\bigl( I_{d}^{r(f)}\setminus A_{d}(f) \bigr) = O\bigl(d^{r(f)-1}\bigr).
    \]
    Moreover,
    \[
        0 \leq N_{f}(a) \leq \operatorname{mult}(f),
    \]
    and each factor $\psi_{1}^{\epsilon_{k}}(a_{k};d)$ is $O(d^{-1})$ coefficientwise, uniformly in $a_{k}\in I_{d}$, whereas $\psi_{0}^{\pm}$ are coefficientwise bounded. Hence $w_{f}(a)=O\bigl(d^{-r(f)}\bigr)$ uniformly in $a$, and therefore
    \[
        \sum_{\sigma\in S_{d}(f)} W_{d}^{\mathrm{lead}}(\sigma)
        = \operatorname{mult}(f) \sum_{a\in I_{d}^{r(f)}}w_{f}(a) + O(d^{-1}).
    \]
    By the definition of $w_{f}$,
    \[
        \sum_{a\in I_{d}^{r(f)}}w_{f}(a)
        = (\psi_{0}^{+})^{s_{+}(f)} (\psi_{0}^{-})^{s_{-}(f)} \prod_{k=1}^{r(f)}
        \left(
            \sum_{a_{k}\in I_{d}}
            \psi_{1}^{\epsilon_{k}}(a_{k};d)
        \right).
    \]
    By definition,
    \[
        \psi_{0}^{+}=t, \qquad
        \psi_{0}^{-}=t^{-1},
    \]
    while
    \[
        \sum_{a=0}^{d-1}\psi_{1}^{+}(a;d) = 1-t, \qquad
        \sum_{a=0}^{d-1}\psi_{1}^{-}(a;d) = 1-t^{-1}.
    \]
    Thus
    \[
        \sum_{a\in I_{d}^{r(f)}}w_{f}(a)
        = t^{s_{+}(f)} (t^{-1})^{s_{-}(f)} \prod_{k=1}^{r(f)} \bigl(1-t^{\epsilon_{k}}\bigr),
    \]
    which proves the claim.
\end{proof}

\begin{thm}[Melvin--Morton--Rozansky conjecture]
    One has
    \[
        \lim_{d\to\infty} \left. \frac{J_{\mathfrak{sl}_2}(K,V_{d+1})}{[d+1]_{q}} \right|_{q=e^{h/d}}
        = \Delta_{K}(e^{-2h})^{-1}.
    \]
    The limit is coefficientwise in $h$.
\end{thm}

\begin{proof}
    Fix $m\geq 0$. Since
    \[
        S_{d} = \coprod_{f\in\mathcal{F}(K_{D^{\ast}})_{\leq d}} S_{d}(f),
    \]
    we have
    \[
        Z_{d}^{\mathrm{lead}}
        = \sum_{f\in\mathcal{F}(K_{D^{\ast}})_{\leq d}}
        \sum_{\sigma\in S_{d}(f)} W_{d}^{\mathrm{lead}}(\sigma).
    \]
    By Remark~\ref{rem:coefficientwise-finiteness}, only flows with $r(f)\leq m$ can contribute to the coefficients of degree at most $m$. There are only finitely many such flows, and for all sufficiently large $d$ all of them belong to $\mathcal{F}(K_{D^{\ast}})_{\leq d}$.
   
    By the definition of $w_{t}$, the main term in Proposition~\ref{prop:fixed-flow-evaluation} is $\operatorname{mult}(f)w_{t}(f)$. Hence
    \[
        Z_{d}^{\mathrm{lead}}
        = \sum_{f\in\mathcal{F}(K_{D^{\ast}})_{\leq d}} \operatorname{mult}(f)w_{t}(f) + O(d^{-1})
    \]
    coefficientwise up to degree $m$.
    
    For such $d$, the restriction $f\in\mathcal{F}(K_{D^{\ast}})_{\leq d}$ may be omitted in these coefficients. Therefore
    \[
        \sum_{f\in\mathcal{F}(K_{D^{\ast}})_{\leq d}} \operatorname{mult}(f)w_{t}(f)
        = \sum_{f\in\mathcal{F}(K_{D^{\ast}})} \operatorname{mult}(f)w_{t}(f)
    \]
    in coefficients of degree at most $m$. By Lemma~\ref{lem:zeta_flow},
    \[
        \sum_{f\in\mathcal{F}(K_{D^{\ast}})} \operatorname{mult}(f)w_{t}(f)
        = \zeta(K_{D^{\ast}},w_{t})
        = \Delta_{D^{\ast}}(t)^{-1}.
    \]
    Since $m$ was arbitrary,
    \[
        Z_{d}^{\mathrm{lead}} = \Delta_{D^{\ast}}(t)^{-1} + O(d^{-1})
    \]
    coefficientwise.
    Corollary~\ref{cor:rt-lead} now gives
    \[
        Z_{d}^{RT} 
        = t^{-\frac{\operatorname{rot}(D^{\ast})}{2}} \Delta_{D^{\ast}}(t)^{-1} + O(d^{-1})
        = \Delta_{K}(t)^{-1} + O(d^{-1}).
    \]
    Therefore
    \begin{align*}
        \lim_{d\to\infty}\left.\frac{J_{\mathfrak{sl}_2}(K,V_{d+1})}{[d+1]_{q}}\right|_{q=e^{h/d}}
        = \lim_{d\to\infty}\left.Z_{d}^{RT}\right|_{q=e^{h/d}}
        = \left. t^{-\frac{\operatorname{rot}(D^{\ast})}{2}} \Delta_{D^{\ast}}(t)^{-1} \right|_{t=e^{-2h}}
        = \Delta_{K}(e^{-2h})^{-1}.
    \end{align*}
\end{proof}

%%%%%%%%%%%%%%%%%%%%%%%%%%%%%%%%%%%%%%%%%%%%%%%%%%%%%%%%%%%
%% \S4
%%%%%%%%%%%%%%%%%%%%%%%%%%%%%%%%%%%%%%%%%%%%%%%%%%%%%%%%%%%

\section{The MMR conjecture for \texorpdfstring{$\mathcal{U}_{q}(\mathfrak{g})$}{Uq(g)}}
\label{sec:semisimple}

\subsection{The state sum and normalization}
\label{subsec:semisimple-state-sum}

We continue with the fixed diagram $D^{\ast}$, the corresponding $(1,1)$--tangle $T$, and the graph $K_{D^{\ast}}$ of \S\ref{sec:zeta_alex}.

Let $\mathfrak{g}$ be a complex semisimple Lie algebra, let $\mathfrak{h}$ be a Cartan subalgebra, and fix a set $\Delta_{+}$ of positive roots. Let $P$ be the weight lattice. On each simple summand of $\mathfrak{g}$, normalize the invariant bilinear form so that every short root has squared length $2$.
Choose an integer $M\geq 1$ such that
\[
    (P,P)\subset \frac{1}{M}\mathbb{Z}.
\]
Put
\[
    \mathbb{K}:=\mathbb{C}\bigl(q^{1/M}\bigr),
\]
and define 
\[
    \mathscr{A}
    :=
    \left\{
        \frac{f(q^{1/M})}{g(q^{1/M})}
        \,\middle|\,
        f,g\in\mathbb{C}[q^{\pm 1/M}],
        \quad
        g(1)\neq 0
    \right\}.
\]
Thus $\mathscr{A}$ is the local subring of $\mathbb{K}$ consisting of functions regular at $q^{1/M}=1$. We regard $\mathcal{U}_{q}(\mathfrak{g})$ and its finite-dimensional type-one modules as defined over $\mathbb{K}$.

Let $\lambda$ be a dominant integral weight, and let $V_{\lambda}$ be the irreducible highest weight $\mathcal{U}_{q}(\mathfrak{g})$--module with highest weight vector $v_{\lambda}$.
Choose a convex ordering of $\Delta_{+}$ and the corresponding Lusztig root vectors $E_{\alpha}$ and $F_{\alpha}$.
Choose an $\mathscr{A}$--form
\[
    V_{\lambda,\mathscr{A}} \subset V_{\lambda}
\]
with $V_{\lambda,\mathscr{A}}[\lambda] = \mathscr{A}v_{\lambda}$, stable under the chosen root vectors $E_{\alpha}$ and $F_{\alpha}$, and such that every weight space is finite free over $\mathscr{A}$. Such a form is obtained by extension of scalars from the standard integral form; see, e.g., \cite{Jantzen1996lectures}. All bases of $V_{\lambda,\mathscr{A}}$ below are $\mathscr{A}$--bases, and scalar extension to $\mathbb{K}$ is understood.

For $\alpha\in\Delta_{+}$, put
\[
    d_{\alpha} := \frac{(\alpha,\alpha)}{2}, \qquad
    q_{\alpha} := q^{d_{\alpha}}, \qquad
    \lambda_{\alpha} := \langle\lambda,\alpha^{\vee}\rangle.
\]
Then
\[
    (\lambda,\alpha) = d_{\alpha}\lambda_{\alpha}.
\]

For the $\mathcal{U}_{q_{\alpha}}(\mathfrak{sl}_{2})$--subalgebra corresponding to $\alpha$, generated by $E_{\alpha}$, $F_{\alpha}$, and $K_{\alpha}^{\pm1}$, the vector $v_{\lambda}$ is a highest weight vector of highest weight $\lambda_{\alpha}$, with
\[
    K_{\alpha}v_{\lambda} = q_{\alpha}^{\lambda_{\alpha}}v_{\lambda}.
\]
Hence
\[
    E_{\alpha}F_{\alpha}v_{\lambda}
    = [E_{\alpha},F_{\alpha}]v_{\lambda}
    = \frac{K_{\alpha}-K_{\alpha}^{-1}} {q_{\alpha}-q_{\alpha}^{-1}}v_{\lambda}
    = [\lambda_{\alpha}]_{q_{\alpha}}v_{\lambda}.
\]

Put
\[
    \Delta_{\lambda,+}
    :=
    \left\{
        \alpha\in\Delta_{+}
        \,\middle|\,
        \lambda_{\alpha}>0
    \right\}, \qquad
    x_{\alpha}:=F_{\alpha}v_{\lambda} \quad
    (\alpha\in\Delta_{\lambda,+}).
\]
Since $\lambda$ is dominant, $\alpha\notin\Delta_{\lambda,+}$ implies $\lambda_{\alpha}=0$, and hence
\[
    F_{\alpha}v_{\lambda}=0 \qquad
    (\alpha\notin\Delta_{\lambda,+}).
\]

Throughout this section, we put
\[
    q=e^{h/d}, \qquad
    q^{1/M}=e^{h/(Md)}, \qquad
    I_{d}:=\{0,1,\ldots,d-1\}.
\]

\begin{lem}\label{lem:semisimple-basis}
    There exists a weight basis $\mathcal{B}_{\lambda}$ of $V_{\lambda,\mathscr{A}}$ with the following properties:
    \[
        v_{\lambda}\in\mathcal{B}_{\lambda},\qquad
        x_{\alpha}\in\mathcal{B}_{\lambda} \quad
        (\alpha\in\Delta_{\lambda,+}),
    \]
    and, for every $\gamma\in\Delta_{+}$ and every $b\in\mathcal{B}_{\lambda}$,
    \[
        [v_{\lambda}]\,E_{\gamma}b \neq 0
        \Longrightarrow \gamma\in\Delta_{\lambda,+} \ \text{and}\ b=x_{\gamma}.
    \]
    Here $[v_{\lambda}]\,u$ denotes the coefficient of $v_{\lambda}$ in the weight-$\lambda$ component of $u$.
\end{lem}

\begin{proof}
    Fix $\gamma\in\Delta_{+}$, and consider the $\mathscr{A}$--linear map
    \[
        p_{\gamma} : V_{\lambda,\mathscr{A}}[\lambda-\gamma] \longrightarrow \mathscr{A}, \qquad
        p_{\gamma}(b):=[v_{\lambda}]\,E_{\gamma}b.
    \]

    Suppose first that $\gamma\in\Delta_{\lambda,+}$. Then $x_{\gamma}=F_{\gamma}v_{\lambda}\neq 0$, and $p_{\gamma}(x_{\gamma}) = [\lambda_{\gamma}]_{q_{\gamma}}$.

    Since this element specializes to the nonzero integer $\lambda_{\gamma}$ at $q=1$, it is a unit in $\mathscr{A}$. Hence
    \[
        V_{\lambda,\mathscr{A}}[\lambda-\gamma]
        = \mathscr{A}x_{\gamma} \oplus \ker p_{\gamma}.
    \]

    Since
    \[
        \mathscr{A} = \mathbb{C}[q^{1/M}]_{(q^{1/M}-1)}
    \]
    is a PID, every submodule of a finite free $\mathscr{A}$--module is free. Hence $\ker p_{\gamma}$ is free. Choose an $\mathscr{A}$--basis of $\ker p_{\gamma}$; together with $x_{\gamma}$, it gives an $\mathscr{A}$--basis of $V_{\lambda,\mathscr{A}}[\lambda-\gamma]$.

    Suppose next that $\gamma\notin\Delta_{\lambda,+}$, so that $\lambda_{\gamma} = \langle\lambda,\gamma^{\vee}\rangle = 0$.
    Let $\{\alpha_{i}\}_{i\in I}$ be the set of simple roots, and write
    \[
        \gamma=\sum_{i\in I}m_{i}\alpha_{i}, \qquad
        \gamma^{\vee} = \sum_{i\in I}c_{i}\alpha_{i}^{\vee}.
    \]
    Here $m_{i},c_{i}\geq 0$, and the supports of $\gamma$ and $\gamma^{\vee}$ agree. Since $\lambda$ is dominant,
    \[
        0 = \langle\lambda,\gamma^{\vee}\rangle 
        = \sum_{i\in I} c_{i}\langle\lambda,\alpha_{i}^{\vee}\rangle
    \]
    implies
    \[
        \langle\lambda,\alpha_{i}^{\vee}\rangle = 0 \qquad
        \text{whenever }m_{i}>0.
    \]
    Consequently, with $F_{i} := F_{\alpha_{i}}$,
    \[
        F_{i}v_{\lambda} = 0 \qquad
        \text{whenever }m_{i} > 0.
    \]
    By the triangular decomposition, $V_{\lambda}[\lambda-\gamma]$ is spanned by vectors
    \[
        F_{i_{1}}\cdots F_{i_{r}}v_{\lambda} \qquad
        \text{with }
        \alpha_{i_{1}}+\cdots+\alpha_{i_{r}}=\gamma.
    \]
    For every such monomial, the rightmost factor $F_{i_{r}}$ satisfies $m_{i_{r}}>0$, and hence $F_{i_{r}}v_{\lambda}=0$.
    Therefore $V_{\lambda}[\lambda-\gamma]=0$, and hence also
    \[
        V_{\lambda,\mathscr{A}}[\lambda-\gamma]=0.
    \]

    Take $v_{\lambda}$ as the basis of the highest weight space $V_{\lambda,\mathscr{A}}[\lambda]$. For each $\gamma\in\Delta_{\lambda,+}$, use the basis of $V_{\lambda,\mathscr{A}}[\lambda-\gamma]$ chosen above, and choose arbitrary $\mathscr{A}$--bases of the remaining nonzero weight spaces. The union of these bases is the required weight basis $\mathcal{B}_{\lambda}$.
\end{proof}

With the chosen convex ordering, the universal $\mathcal{R}$--matrix is
\[
    \mathcal{R}
    = q^{\Omega_{\mathfrak{h}}}
    \prod_{\alpha\in\Delta_{+}}^{\longrightarrow} \mathcal{R}_{\alpha}, \qquad
    \mathcal{R}_{\alpha}
    = \sum_{n\geq 0} q_{\alpha}^{\frac{1}{2}n(n-1)}
    \frac{(q_{\alpha}-q_{\alpha}^{-1})^{n}}
         {[n]_{q_{\alpha}}!}
    E_{\alpha}^{n}\otimes F_{\alpha}^{n}.
\]
Here $q^{\Omega_{\mathfrak{h}}}$ acts on a tensor product of weight vectors of weights $\mu$ and $\nu$ by the scalar $q^{(\mu,\nu)}$. Moreover,
\[
    \mathcal{R}^{-1}
    =
    \left(
        \prod_{\alpha\in\Delta_{+}}^{\longleftarrow}
        \mathcal{R}_{\alpha}^{-1}
    \right)
    q^{-\Omega_{\mathfrak{h}}},
\]
where
\[
    \mathcal{R}_{\alpha}^{-1}
    = \sum_{n\geq 0} (-1)^{n} q_{\alpha}^{-\frac{1}{2}n(n-1)}
    \frac{(q_{\alpha}-q_{\alpha}^{-1})^{n}}{[n]_{q_{\alpha}}!} E_{\alpha}^{n}\otimes F_{\alpha}^{n}.
\]
For each $\alpha\in\Delta_{+}$, we call $\mathcal{R}_{\alpha}$ and $\mathcal{R}_{\alpha}^{-1}$ the \emph{root factors} corresponding to $\alpha$.
The weight space of $V_{\lambda}^{\otimes d}$ of weight $d\lambda$ is one-dimensional, spanned by $v_{\lambda}^{\otimes d}$. Hence $V_{d\lambda}$ occurs in $V_{\lambda}^{\otimes d}$ with multiplicity one. We therefore fix the $\mathcal{U}_{q}(\mathfrak{g})$--module embedding $\iota_{d,\lambda}: V_{d\lambda} \hookrightarrow V_{\lambda}^{\otimes d}$ by requiring
\[
    \iota_{d,\lambda}(v_{d\lambda}) = v_{\lambda}^{\otimes d}.
\]
As in \S\ref{subsec:sl2_setup}, we write $\mathcal{F}(T;V)$ for the operator assigned to $T$ coloured by $V$. By Schur's lemma,
\[
    \mathcal{F}(T;V_{d\lambda})
    = Z_{d,\lambda}^{RT}\, \operatorname{id}_{V_{d\lambda}}
\]
for some scalar $Z_{d,\lambda}^{RT}$. With our normalization,
\[
    Z_{d,\lambda}^{RT}
    = \frac{J_{\mathfrak{g}}(K,V_{d\lambda})}{\operatorname{dim}_{q}(V_{d\lambda})}.
\]
Naturality with respect to $\iota_{d,\lambda}$ gives
\[
    \mathcal{F}(T;V_{\lambda}^{\otimes d}) v_{\lambda}^{\otimes d}
    = Z_{d,\lambda}^{RT} v_{\lambda}^{\otimes d}.
\]

We use the same notation $\mathcal{R}^{\epsilon}$ for its action on $V_{\lambda}\otimes V_{\lambda}$. For $\epsilon\in\{+1,-1\}$, set $\overline{\mathcal{R}}^{\epsilon} := q^{-\epsilon(\lambda,\lambda)} \mathcal{R}^{\epsilon}$. Then
\[
    \overline{\mathcal{R}}^{\epsilon} (v_{\lambda}\otimes v_{\lambda})
    = v_{\lambda}\otimes v_{\lambda}.
\]
As in \S\ref{subsec:sl2_setup}, let $V_{\lambda}^{(u,i)}$ and $V_{\lambda}^{(o,j)}$ denote the tensor factors with indices $i$ and $j$ on the under-strand and the over-strand, respectively.
For $(i,j)\in I_{d}\times I_{d}$, let
\[
    \overline{\mathcal{R}}_{u_{i};o_{j}}^{\epsilon}
    \in \operatorname{End} \left( V_{\lambda}^{\otimes d} \otimes V_{\lambda}^{\otimes d} \right)
\]
be the operator that acts as $\overline{\mathcal{R}}^{\epsilon}$ on $V_{\lambda}^{(u,i)}\otimes V_{\lambda}^{(o,j)}$ and as the identity on all other tensor factors.

At a positive crossing, the operators $\overline{\mathcal{R}}_{u_{i};o_{j}}^{+1}$ are applied in the order of the pairs $(i,j)$ used in \S\ref{subsec:sl2_setup}; at a negative crossing, the operators $\overline{\mathcal{R}}_{u_{i};o_{j}}^{-1}$ are applied in the reverse order. As in \S\ref{subsec:sl2_setup}, we call each occurrence of one of these operators in the ordered product an elementary factor.

Replacing, at a crossing of sign $\epsilon$, each copy of $\mathcal{R}^{\epsilon}$ by $\overline{\mathcal{R}}^{\epsilon}$ multiplies the operator assigned to that crossing by $q^{-\epsilon d^{2}(\lambda,\lambda)}$.
Hence the product of these scalar factors over all crossings is
\[
    q^{-d^{2}(\lambda,\lambda)\operatorname{wr}(D^{\ast})} = 1.
\]
Thus this replacement does not change $Z_{d,\lambda}^{RT}$.

Put
\[
    \mathcal{B}_{d,\lambda}
    :=
    \left\{
        b_{0}\otimes\cdots\otimes b_{d-1}
        \,\middle|\,
        b_{a}\in\mathcal{B}_{\lambda}
    \right\}.
\]

Let $\mathcal{B}_{d,\lambda}^{\vee}$ be the basis of $(V_{\lambda}^{\otimes d})^{*}$ dual to $\mathcal{B}_{d,\lambda}$. For $b\in\mathcal{B}_{d,\lambda}$, let $b^{\vee}\in\mathcal{B}_{d,\lambda}^{\vee}$ denote the corresponding dual basis vector. As in \S\ref{subsec:sl2_setup}, when $b$ is assigned to an oriented segment of the diagram, it represents $b$ on a downward-oriented segment and $b^{\vee}$ on an upward-oriented segment.

Define $S_{d,\lambda}$ as in \S\ref{subsec:sl2_setup}, with $\mathcal{B}_{d}$ replaced by $\mathcal{B}_{d,\lambda}$ and $x_{+}^{\otimes d}$ by $v_{\lambda}^{\otimes d}$. For $\sigma\in S_{d,\lambda}$, let $W_{d,\lambda}^{\mathrm{cross}}(\sigma)$ be the product of the matrix coefficients at all crossings, computed using the ordered products of the operators $\overline{\mathcal{R}}_{u_{i};o_{j}}^{\epsilon}$, and let $W_{d,\lambda}^{\mathrm{cap/cup}}(\sigma)$ be the product of the scalar contributions from all caps and cups. Put
\[
    W_{d,\lambda}^{RT}(\sigma)
    := W_{d,\lambda}^{\mathrm{cap/cup}}(\sigma) W_{d,\lambda}^{\mathrm{cross}}(\sigma).
\]
Then
\[
    Z_{d,\lambda}^{RT}
    = \sum_{\sigma\in S_{d,\lambda}}  W_{d,\lambda}^{RT}(\sigma).
\]

Define
\[
    Z_{d,\lambda}^{\mathrm{cross}}
    := \sum_{\sigma\in S_{d,\lambda}} W_{d,\lambda}^{\mathrm{cross}}(\sigma).
\]
At a cap or cup, the coefficient involving $b\in\mathcal{B}_{d,\lambda}$ and $c^{\vee}\in\mathcal{B}_{d,\lambda}^{\vee}$ is nonzero only when $b=c$; when $b=c$, it is a scalar determined by the type of cap or cup and the weight of $b$. Hence the two segments adjacent to a cap or cup must carry the same label $b\in\mathcal{B}_{d,\lambda}$, and the cap or cup contributes this scalar.
These scalar contributions are not included in $Z_{d,\lambda}^{\mathrm{cross}}$.

%%%%%%%%%%%%%%%%%%%%%%%%%%%%%%%%%%%%%%%%%%%%%%%%%%%%%%%%%%%%%%%%

%%%%%%%%%%%%%%%%%%%%%%%%%%%%%%%%%%%%%%%%%%%%%%%%%%%%%%%%%%%%%%%%

\subsection{Reduction to \texorpdfstring{$E_{\alpha}\otimes F_{\alpha}$}{EF} summands}

Fix $\sigma\in S_{d,\lambda}$. As in \S\ref{sec:sl2}, using the ordered products of elementary factors and the sums defining $\mathcal{R}_{\alpha}$ and $\mathcal{R}_{\alpha}^{-1}$, we write $W_{d,\lambda}^{\mathrm{cross}}(\sigma)$ as a finite sum.

We index its nonzero summands by a finite set $\mathcal{E}_{d,\lambda}(\sigma)$. For $\eta\in\mathcal{E}_{d,\lambda}(\sigma)$, let $W_{d,\lambda}^{\mathrm{cross}}(\eta)$ denote the corresponding summand. Then
\[
    W_{d,\lambda}^{\mathrm{cross}}(\sigma)
    = \sum_{\eta\in\mathcal{E}_{d,\lambda}(\sigma)} W_{d,\lambda}^{\mathrm{cross}}(\eta).
\]
For an elementary factor $e$ and $\alpha\in\Delta_{+}$, let $n_{e,\alpha}(\eta)\geq 0$ be the integer $n$ such that the summand selected from $\mathcal{R}_{\alpha}$ or $\mathcal{R}_{\alpha}^{-1}$ at $e$, according to the sign of the crossing, is proportional to $E_{\alpha}^{n}\otimes F_{\alpha}^{n}$.
Put
\[
    r_{e}(\eta) := \sum_{\alpha\in\Delta_{+}}n_{e,\alpha}(\eta), \qquad
    r(\eta) := \sum_{e} r_{e}(\eta),
\]
where the second sum is over all elementary factors.
Let $s(\eta)$ be the number of elementary factors $e=\overline{\mathcal{R}}_{u_{i};o_{j}}^{\epsilon}$ such that $r_{e}(\eta)>0$ and, in the summand indexed by $\eta$, the vector in $V_{\lambda}^{(o,j)}$ immediately before $e$ is applied is $v_{\lambda}$.
For $r,s\in\mathbb{Z}_{\geq 0}$, put
\[
    \mathcal{E}_{d,\lambda;r,s}
    :=
    \left\{
        (\sigma,\eta)
        \,\middle|\,
        \sigma\in S_{d,\lambda},
        \ \eta\in\mathcal{E}_{d,\lambda}(\sigma),
        \ r(\eta)=r,
        \ s(\eta)=s
    \right\}.
\]

\begin{lem}\label{lem:semisimple-order}
    For every $(\sigma,\eta)\in\mathcal{E}_{d,\lambda;r,s}$,
    \[
        W_{d,\lambda}^{\mathrm{cross}}(\eta) \in h^{r}\mathbb{C}[[h]].
    \]
    Moreover, for every $m\geq 0$, there exists a constant $C_{T,m}>0$, independent of $d$, $\sigma$, and $\eta$, such that
    \[
        \left|
            W_{d,\lambda}^{\mathrm{cross}}(\eta)
        \right|_{(m)}
        \leq C_{T,m}d^{-r}.
    \]
\end{lem}

\begin{proof}
    For each elementary factor $e$ and $\alpha\in\Delta_{+}$, the scalar coefficient of the summand selected from $\mathcal{R}_{\alpha}$ or $\mathcal{R}_{\alpha}^{-1}$ at $e$ contains the factor $(q_{\alpha}-q_{\alpha}^{-1})^{n_{e,\alpha}(\eta)}$.
    Since $q_{\alpha}=e^{d_{\alpha}h/d}$,
    \[
        (q_{\alpha}-q_{\alpha}^{-1})^{n} \in h^{n}\mathbb{C}[[h]].
    \]
    The factors $q_{\alpha}^{\pm\frac{1}{2}n(n-1)}$,  $[n]_{q_{\alpha}}!^{-1}$, and the coefficients of $E_{\alpha}^{n}$ and $F_{\alpha}^{n}$ with respect to $\mathcal{B}_{\lambda}$ are regular at $q=1$. Since
    \[
        r(\eta)
        = \sum_{e}\sum_{\alpha\in\Delta_{+}} n_{e,\alpha}(\eta)
        = r,
    \]
    it follows that $W_{d,\lambda}^{\mathrm{cross}}(\eta) \in h^{r}\mathbb{C}[[h]]$.
    Fix $m\geq0$. If $r>m$, then $\left|W_{d,\lambda}^{\mathrm{cross}}(\eta)\right|_{(m)} = 0$, so assume $r\leq m$. For $n\leq m$,
    \[
        \left|
            (q_{\alpha}-q_{\alpha}^{-1})^{n}
        \right|_{(m)}
        = O(d^{-n})
    \]
    uniformly in $\alpha\in\Delta_{+}$. Moreover,
    \[
        \#\left\{
            (e,\alpha)
            \,\middle|\,
            n_{e,\alpha}(\eta)>0
        \right\}
        \leq r\leq m.
    \]
    Since $\Delta_{+}$ and $\mathcal{B}_{\lambda}$ are finite, the other factors displayed above are uniformly bounded in coefficients of degree at most $m$. Hence the product of the contributions from the root factors is $O(d^{-r})$ in these coefficients.

    It remains to consider $q^{\pm\Omega_{\mathfrak{h}}}q^{\mp(\lambda,\lambda)}$. Among all elementary factors $e=\overline{\mathcal{R}}_{u_{i};o_{j}}^{\epsilon}$ with $r_{e}(\eta)>0$, at most $2r$ values in $I_{d}$ occur as $i$ or $j$.
    If, at a crossing, a tensor factor indexed by $a\in I_{d}$ carries a basis vector different from $v_{\lambda}$, then $a=i$ or $a=j$ for some elementary factor $e=\overline{\mathcal{R}}_{u_{i};o_{j}}^{\epsilon}$ with $r_{e}(\eta)>0$: the boundary values are $v_{\lambda}^{\otimes d}$, the labels on the two adjacent segments at a cap or cup correspond to the same element of $\mathcal{B}_{d,\lambda}$, and an elementary factor with $r_{e}(\eta)=0$ does not change the basis vectors.
    For weights $\mu$ and $\nu$, the scalar contributed by $q^{\pm\Omega_{\mathfrak{h}}}q^{\mp(\lambda,\lambda)}$ at a crossing of sign $\epsilon$ is
    \[
        q^{\epsilon((\mu,\nu)-(\lambda,\lambda))},
    \]
    which is $1$ when $\mu=\nu=\lambda$. Hence, at each crossing, this scalar can differ from $1$ only for a pair $(i,j)\in I_{d}\times I_{d}$ for which $i$ or $j$ is one of the at most $2r$ indices above. There are therefore at most $4rd$ such pairs at each crossing. Since $r\leq m$ and the number of crossings of $D^{\ast}$ is fixed, the total number of these scalars that differ from $1$ is at most $C_{T,m}d$.
    The weights of $V_{\lambda}$ form a finite set, so the product of these scalars has the form $q^{\kappa}$, where $|\kappa|\leq C_{T,m}d$.
    Since $q=e^{h/d}$, $\left|q^{\kappa}\right|_{(m)}$ is bounded independently of $d$, $\sigma$, and $\eta$. Combining this with the $O(d^{-r})$ estimate above gives
    \[
        \left| W_{d,\lambda}^{\mathrm{cross}}(\eta) \right|_{(m)}
        \leq C_{T,m}d^{-r}.
    \]
\end{proof}

\begin{lem}\label{lem:semisimple-index-count}
    Fix $r,s\in\mathbb{Z}_{\geq0}$. Then
    \[
        \#\mathcal{E}_{d,\lambda;r,s} = O(d^{s}).
    \]
    Moreover, for every $m\geq0$,
    \[
        \left|
            \sum_{(\sigma,\eta)\in\mathcal{E}_{d,\lambda;r,s}}
            W_{d,\lambda}^{\mathrm{cross}}(\eta)
        \right|_{(m)}
        = O(d^{s-r}).
    \]
\end{lem}

\begin{proof}
    Fix $(\sigma,\eta)\in\mathcal{E}_{d,\lambda;r,s}$. At most $r$ elementary factors satisfy $r_{e}(\eta)>0$.
    Let $e=\overline{\mathcal{R}}_{u_{i};o_{j}}^{\epsilon}$ satisfy $r_{e}(\eta)>0$. The vector in $V_{\lambda}^{(u,i)}$ immediately before $e$ is applied is different from $v_{\lambda}$.
    Otherwise, let $\alpha$ be the first root, in the order in which the root factors act at $e$, such that $n_{e,\alpha}(\eta)>0$. The $n=0$ summand is selected from every preceding root factor, while
    \[
        E_{\alpha}^{n_{e,\alpha}(\eta)}v_{\lambda}=0,
    \]
    contradicting the fact that $\eta$ indexes a nonzero summand.
    Now fix an edge $e_{0}$ of $\Gamma_{T}$, write
    \[
        \sigma(e_{0}) = b_{0}\otimes\cdots\otimes b_{d-1},
    \]
    and fix $a\in I_{d}$ such that $b_{a}\neq v_{\lambda}$. Move backwards along the oriented tangle while keeping the index $a$ fixed.
    Across a cap or cup, the two adjacent segments carry the same label in $\mathcal{B}_{d,\lambda}$, so the component at index $a$ is unchanged. The component at index $a$ is also unchanged across an elementary factor $e'$ with $r_{e'}(\eta)=0$. Suppose that, in moving backwards with index $a$, an elementary factor $e'=\overline{\mathcal{R}}_{u_{i'};o_{j'}}^{\epsilon'}$ with $r_{e'}(\eta)>0$ is reached.
    If $e'$ is reached along its under-strand, then $a=i'$, and the argument above shows that the vector in $V_{\lambda}^{(u,i')}$ immediately before $e'$ is applied is different from $v_{\lambda}$.
    If $e'$ is reached along its over-strand, then $a=j'$; if the vector in $V_{\lambda}^{(o,j')}$ immediately before $e'$ is applied is different from $v_{\lambda}$, we continue backwards with index $a$.

    Since both boundary edges are assigned $v_{\lambda}^{\otimes d}$, one must eventually reach an elementary factor $e'=\overline{\mathcal{R}}_{u_{i'};o_{j'}}^{\epsilon'}$ with $r_{e'}(\eta)>0$ for which the vector in $V_{\lambda}^{(o,j')}$ immediately before $e'$ is applied is $v_{\lambda}$.
    By the definition of $s(\eta)$, such an elementary factor is counted by $s(\eta)$.

    There are exactly $s$ elementary factors counted by $s(\eta)$. Consider the $s$ values of $j$ appearing in these elementary factors. For every elementary factor $e=\overline{\mathcal{R}}_{u_{i};o_{j}}^{\epsilon}$ with $r_{e}(\eta)>0$, the preceding argument shows that $i$ is one of these $s$ values. The value of $j$ is also one of them: either $e$ itself is counted by $s(\eta)$, or the preceding argument applies to the vector in $V_{\lambda}^{(o,j)}$ immediately before $e$ is applied. Thus both $i$ and $j$ are among the $s$ values of $j$ appearing in the elementary factors counted by $s(\eta)$. These $s$ values have at most $d^{s}$ possible choices.

    Since $r$ is fixed and $D^{\ast}$ has finitely many crossings, the possible choices of these elementary factors are bounded independently of $d$. Once these elementary factors are fixed, the number of corresponding pairs $(\sigma,\eta)$ is bounded independently of $d$. Therefore
    \[
        \#\mathcal{E}_{d,\lambda;r,s}=O(d^{s}).
    \]
    By Lemma~\ref{lem:semisimple-order}, for every fixed $m\geq0$, uniformly over $(\sigma,\eta)\in\mathcal{E}_{d,\lambda;r,s}$, $\left| W_{d,\lambda}^{\mathrm{cross}}(\eta) \right|_{(m)} = O(d^{-r})$.    
    Hence
    \[
        \left|
            \sum_{(\sigma,\eta)\in\mathcal{E}_{d,\lambda;r,s}}
            W_{d,\lambda}^{\mathrm{cross}}(\eta)
        \right|_{(m)}
        \leq \sum_{(\sigma,\eta)\in\mathcal{E}_{d,\lambda;r,s}}
        \left| W_{d,\lambda}^{\mathrm{cross}}(\eta) \right|_{(m)}
        = O(d^{s-r}).
    \]
\end{proof}

\begin{prop}[Reduction to $E_{\alpha}\otimes F_{\alpha}$ summands]\label{prop:semisimple-root-reduction}
    The total contribution to $Z_{d,\lambda}^{\mathrm{cross}}$ from summands $W_{d,\lambda}^{\mathrm{cross}}(\eta)$ that fail the following condition is $O(d^{-1})$ coefficientwise: for every elementary factor $e$ with $r_{e}(\eta)>0$, one has $r_{e}(\eta)=1$, and, for the unique $\alpha\in\Delta_{+}$ satisfying $n_{e,\alpha}(\eta)=1$, one has $\alpha\in\Delta_{\lambda,+}$ and
    \[
        x_{\alpha}\otimes v_{\lambda} \longmapsto v_{\lambda}\otimes x_{\alpha}
    \]
    at $e$.
\end{prop}

\begin{proof}
    Fix $m\geq0$. By Lemma~\ref{lem:semisimple-order}, only $r(\eta)\leq m$ can contribute to coefficients of degree at most $m$. Lemma~\ref{lem:semisimple-index-count} therefore gives
    \[
        \left|
            \sum_{r=1}^{m}
            \sum_{s=0}^{r-1}
            \sum_{(\sigma,\eta)\in\mathcal{E}_{d,\lambda;r,s}}
            W_{d,\lambda}^{\mathrm{cross}}(\eta)
        \right|_{(m)}
        = O(d^{-1}).
    \]

    Suppose now that $s(\eta)=r(\eta)$. Since
    \[
        s(\eta)
        \leq \#\{e\mid r_{e}(\eta)>0\}
        \leq \sum_{e}r_{e}(\eta)
        = r(\eta),
    \]
    both inequalities are equalities. Hence every elementary factor $e$ with $r_{e}(\eta)>0$ is counted by $s(\eta)$, and
    \[
        r_{e}(\eta)=1 \qquad
        \text{whenever }r_{e}(\eta)>0.
    \]
    For each such $e$, there is therefore a unique $\alpha\in\Delta_{+}$ such that
    \[
        n_{e,\alpha}(\eta)=1, \qquad
        n_{e,\beta}(\eta)=0 \quad
        (\beta\neq\alpha).
    \]
    Write $e=\overline{\mathcal{R}}_{u_{i};o_{j}}^{\epsilon}$. Since $e$ is counted by $s(\eta)$, the vector in $V_{\lambda}^{(o,j)}$ immediately before $e$ is applied is $v_{\lambda}$.
    Since $\eta$ indexes a nonzero summand,
    \[
        F_{\alpha}v_{\lambda}\neq0,
    \]
    so $\alpha\in\Delta_{\lambda,+}$ and $F_{\alpha}v_{\lambda}=x_{\alpha}$.  Thus, for some $b,b_{1}\in\mathcal{B}_{\lambda}$,
    \[
        b\otimes v_{\lambda} \longmapsto b_{1}\otimes x_{\alpha}
    \]
    at $e$.
    For a crossing vertex $c$ of $\Gamma_{T}$, let $\delta(c)$ be the total number of occurrences of $v_{\lambda}$ in the two elements of $\mathcal{B}_{d,\lambda}$ assigned to the edges entering $c$, minus the corresponding number for the edges leaving $c$. The form above shows that $\delta(c) \geq 0$.
    Since the two segments adjacent to each cap or cup carry the same label and both boundary edges carry the label $v_{\lambda}^{\otimes d}$,
    \[
        \sum_{c}\delta(c)=0.
    \]
    Hence $\delta(c)=0$ for every crossing vertex $c$. Since none of the elementary factors can increase the total number of occurrences of $v_{\lambda}$, it follows that $b_{1}=v_{\lambda}$ for every elementary factor with $r_{e}(\eta)>0$. Thus $[v_{\lambda}]\,E_{\alpha}b \neq 0$.
    Lemma~\ref{lem:semisimple-basis} gives $b=x_{\alpha}$. This proves the proposition.
\end{proof}

Put
\[
    Z_{d,\lambda}^{\mathrm{lead}}
    := \sum_{\sigma\in S_{d,\lambda}}
    \sum_{\substack{
        \eta\in\mathcal{E}_{d,\lambda}(\sigma)\\
        s(\eta)=r(\eta)
    }}
    W_{d,\lambda}^{\mathrm{cross}}(\eta).
\]
Then Proposition~\ref{prop:semisimple-root-reduction} gives
\[
    Z_{d,\lambda}^{\mathrm{cross}} = Z_{d,\lambda}^{\mathrm{lead}} + O(d^{-1})
\]
coefficientwise.

%%%%%%%%%%%%%%%%%%%%%%%%%%%%%%%%%%%%%%%%%%%%%%%%%%%%%%%%%%%%%%%%

%%%%%%%%%%%%%%%%%%%%%%%%%%%%%%%%%%%%%%%%%%%%%%%%%%%%%%%%%%%%%%%%

\subsection{Modification of coefficients and \texorpdfstring{$Z_{d,\lambda}^{A}$}{Z(d,A)}}
\label{subsec:semisimple-cartan-gauge}

For every $\alpha\in\Delta_{+}$, put
\[
    t_{\alpha}
    := q^{-2d(\lambda,\alpha)}
    = e^{-2(\lambda,\alpha)h}.
\]

\begin{lem}\label{lem:semisimple-local-root}
    For an elementary factor at a crossing of sign $\epsilon\in\{+1,-1\}$, the matrix coefficients that can occur in summands contributing to $Z_{d,\lambda}^{\mathrm{lead}}$ are determined by the following formulas, where $\alpha,\beta\in\Delta_{\lambda,+}$:
    \begin{align*}
        v_{\lambda}\otimes v_{\lambda}
        &\longmapsto v_{\lambda}\otimes v_{\lambda}, \\
        x_{\alpha}\otimes v_{\lambda}
        &\longmapsto 
        q^{-\epsilon(\lambda,\alpha)} x_{\alpha}\otimes v_{\lambda}
        + \left(1-q^{-2\epsilon(\lambda,\alpha)}\right) v_{\lambda}\otimes x_{\alpha}, \\
        v_{\lambda}\otimes x_{\beta}
        &\longmapsto q^{-\epsilon(\lambda,\beta)} v_{\lambda}\otimes x_{\beta}, \\
        x_{\alpha}\otimes x_{\beta}
        &\longmapsto q^{-\epsilon(\lambda,\alpha)}q^{-\epsilon(\lambda,\beta)} q^{\epsilon(\alpha,\beta)} x_{\alpha}\otimes x_{\beta}.
    \end{align*}
\end{lem}

\begin{proof}
    If the identity is selected from every root factor at an elementary factor, then the corresponding coefficient on weight vectors of weights $\mu$ and $\nu$ is
    \[
        q^{\epsilon\left((\mu,\nu)-(\lambda,\lambda)\right)}.
    \]
    For the four cases in the statement, the corresponding pairs $(\mu, \nu)$ are, respectively, $(\lambda, \lambda), (\lambda-\alpha,\lambda), (\lambda, \lambda-\beta)$, and $(\lambda-\alpha,\lambda-\beta)$.
    By Proposition~\ref{prop:semisimple-root-reduction}, if $r_{e}(\eta)>0$ at an elementary factor $e$ in a summand contributing to $Z_{d,\lambda}^{\mathrm{lead}}$, then there exists $\gamma\in\Delta_{\lambda,+}$ such that
    \[
        x_{\gamma} \otimes v_{\lambda} \longmapsto v_{\lambda}\otimes x_{\gamma}
    \]
    at $e$.
    It remains to compute the coefficient of this transition.
    Since $r_{e}(\eta)=1$, the summand selected from $\mathcal{R}_{\gamma}$ or $\mathcal{R}_{\gamma}^{-1}$ is $(q_{\gamma}-q_{\gamma}^{-1})E_{\gamma}\otimes F_{\gamma}$ or $-(q_{\gamma}-q_{\gamma}^{-1})E_{\gamma}\otimes F_{\gamma}$, respectively.
    Thus, for a crossing of sign $\epsilon$, this scalar coefficient is $\epsilon\left(q_{\gamma}-q_{\gamma}^{-1}\right)$.
    Moreover,
    \[
        E_{\gamma}x_{\gamma} = [\lambda_{\gamma}]_{q_{\gamma}}v_{\lambda}, \qquad
        F_{\gamma}v_{\lambda} = x_{\gamma}.
    \]
    For $\epsilon=+1$, $q^{-(\lambda,\lambda)}q^{\Omega_{\mathfrak h}}$ acts after $E_{\gamma}\otimes F_{\gamma}$, whereas for $\epsilon=-1$, $q^{(\lambda,\lambda)}q^{-\Omega_{\mathfrak h}}$ acts before $E_{\gamma}\otimes F_{\gamma}$. In both cases the resulting scalar is $q_{\gamma}^{-\epsilon\lambda_{\gamma}}$.
    Hence the coefficient of the transition is
    \[
        \epsilon q_{\gamma}^{-\epsilon\lambda_{\gamma}}
        \left(q_{\gamma}-q_{\gamma}^{-1}\right)
        [\lambda_{\gamma}]_{q_{\gamma}}
        = 1-q_{\gamma}^{-2\epsilon\lambda_{\gamma}}
        = 1-q^{-2\epsilon(\lambda,\gamma)}.
    \]
    Together with the preceding coefficients, this proves the lemma.
\end{proof}

We now modify the matrix coefficients in Lemma~\ref{lem:semisimple-local-root} as follows. For $\epsilon\in \{+1, -1\}$, define $ r_{A}^{\epsilon}$ by
\begin{align*}
    r_{A}^{\epsilon}(v_{\lambda}\otimes v_{\lambda})
    &= v_{\lambda}\otimes v_{\lambda}, \\
    r_{A}^{\epsilon}(x_{\alpha}\otimes v_{\lambda})
    &= q^{-2\epsilon(\lambda,\alpha)} x_{\alpha}\otimes v_{\lambda}
    + \left(1-q^{-2\epsilon(\lambda,\alpha)}\right) v_{\lambda}\otimes x_{\alpha}, \\
    r_{A}^{\epsilon}(v_{\lambda}\otimes x_{\beta})
    &= v_{\lambda}\otimes x_{\beta}, \\
    r_{A}^{\epsilon}(x_{\alpha}\otimes x_{\beta})
    &= q^{-2\epsilon(\lambda,\alpha)} x_{\alpha}\otimes x_{\beta}.
\end{align*}
Let $Z_{d,\lambda}^{A}$ be obtained from $Z_{d,\lambda}^{\mathrm{lead}}$ by replacing, at each elementary factor in the ordered products, the coefficients in Lemma~\ref{lem:semisimple-local-root} by the corresponding coefficients of $r_{A}^{\epsilon}$. The scalar contributions at caps and cups are not included.

\begin{prop}[Comparison with $Z_{d,\lambda}^{A}$]
    \label{prop:semisimple-cartan-redistribution}
    One has
    \[
        Z_{d,\lambda}^{\mathrm{lead}} = Z_{d,\lambda}^{A} + O(d^{-1})
    \]
    coefficientwise.
\end{prop}

\begin{proof}
    Fix $m\geq0$. The argument of Proposition~\ref{prop:kappa-bound}, applied separately to each $\alpha\in\Delta_{\lambda,+}$, shows that the changes involving $x_{\alpha}$ and $v_{\lambda}$ cancel in the product of the coefficients for each summand.

    It remains to consider
    \[
        x_{\alpha}\otimes x_{\beta} \longmapsto x_{\alpha}\otimes x_{\beta}.
    \]
    The ratio of the new coefficient to the original one is
    \[
        q^{\epsilon\left(
            (\lambda,\beta)-(\lambda,\alpha)-(\alpha,\beta)
        \right)}.
    \]
    By Lemma~\ref{lem:semisimple-order}, only $r\leq m$ can contribute to coefficients of degree at most $m$. The proof of that lemma shows that at most $2r\leq 2m$ tensor-factor indices can carry a vector different from $v_{\lambda}$. Hence the number of such elementary factors in each summand is bounded independently of $d$.
    Since $\Delta_{\lambda,+}$ is finite and $q=e^{h/d}$, the product of the above ratios is $1+O(d^{-1})$ coefficientwise in degrees at most $m$, uniformly over the summands of $Z_{d,\lambda}^{\mathrm{lead}}$.

    Lemmas~\ref{lem:semisimple-order} and~\ref{lem:semisimple-index-count} also show that the sum of $\left| W_{d,\lambda}^{\mathrm{cross}}(\eta) \right|_{(m)}$ over the summands of $Z_{d,\lambda}^{\mathrm{lead}}$ is $O(1)$. Hence
    \[
        Z_{d,\lambda}^{\mathrm{lead}} = Z_{d,\lambda}^{A} + O(d^{-1})
    \]
    in coefficients of degree at most $m$. Since $m$ is arbitrary, the result follows coefficientwise.
\end{proof}

%%%%%%%%%%%%%%%%%%%%%%%%%%%%%%%%%%%%%%%%%%%%%%%%%

%%%%%%%%%%%%%%%%%%%%%%%%%%%%%%%%%%%%%%%%%%%%%%%%%
\subsection{A proof of the MMR conjecture for semisimple Lie algebras}

\begin{prop}[Evaluation of the flow sum] \label{prop:semisimple-flow-evaluation}
    Let
    \[
        \mathcal{F}_{\lambda}(K_{D^{\ast}}) 
        := \prod_{\alpha\in\Delta_{\lambda,+}}\mathcal{F}(K_{D^{\ast}}).
    \]
    One has
    \[
        Z_{d,\lambda}^{A}
        =
        \sum_{\substack{
            (f_{\alpha})_{\alpha\in\Delta_{\lambda,+}} \in \mathcal{F}_{\lambda}(K_{D^{\ast}}) 
        }}
        \prod_{\alpha\in\Delta_{\lambda,+}} \operatorname{mult}(f_{\alpha}) w_{t_{\alpha}}(f_{\alpha}) + O(d^{-1})
    \]
    coefficientwise. Consequently,
    \[
        Z_{d,\lambda}^{A}
        = \prod_{\alpha\in\Delta_{\lambda,+}} \zeta(K_{D^{\ast}}, w_{t_{\alpha}}) + O(d^{-1})
    \]
    coefficientwise.
\end{prop}

\begin{proof}
    For each $\alpha\in\Delta_{\lambda,+}$, the formulas defining $r_{A}^{\epsilon}$ preserve the number of occurrences of $x_{\alpha}$. Hence, for every nonzero summand, counting the occurrences of $x_{\alpha}$ in the same way as $x_{-}$ is counted in the definition of $p_d$ in \S\ref{subsec:lead-evaluation} yields a flow
    \[
        f_{\alpha}\in\mathcal{F}(K_{D^{\ast}}).
    \]
    They necessarily satisfy
    \[
        \sum_{\alpha\in\Delta_{\lambda,+}} f_{\alpha}(v) \leq d \qquad
        (v\in V(K_{D^{\ast}})).
    \]
    Fix a family $(f_{\alpha})_{\alpha\in\Delta_{\lambda,+}}$.
    For each edge of $\Gamma_{T}$ and each index $a\in I_d$, the component at index $a$ of the basis label assigned to that edge can be equal to $x_{\alpha}$ for at most one $\alpha\in\Delta_{\lambda,+}$.
    The counting argument in the proof of Lemma~\ref{lem:bad-set} depends on the chosen indices only through whether two of them coincide, and not on the root labels attached to the corresponding transitions
    \[
        x_{\alpha}\otimes v_{\lambda} \longmapsto v_{\lambda}\otimes x_{\alpha}.
    \]
    Hence the same argument applies to the indices for all $\alpha$ together, with $r(f)$ replaced by
    \[
        \sum_{\alpha\in\Delta_{\lambda,+}} r(f_{\alpha}).
    \]
    For each $\alpha$, the coefficient calculation in the proof of Proposition~\ref{prop:fixed-flow-evaluation} is obtained by replacing $q^{-2}$ with $q^{-2(\lambda,\alpha)}$. Combining these observations with the proof of Proposition~\ref{prop:fixed-flow-evaluation}, the contribution is
    \[
        \prod_{\alpha\in\Delta_{\lambda,+}}
        \operatorname{mult}(f_{\alpha}) w_{t_{\alpha}}(f_{\alpha})
        + O(d^{-1}).
    \]

    For every fixed $m\geq 0$, only finitely many families $(f_{\alpha})_{\alpha\in\Delta_{\lambda,+}}$ can contribute to coefficients of degree at most $m$. For all sufficiently large $d$, all of them satisfy the inequality above. Summing the preceding estimate therefore gives
    \[
        Z_{d,\lambda}^{A}
        = \sum_{\substack{
            (f_{\alpha})_{\alpha\in\Delta_{\lambda,+}} \in \mathcal{F}_{\lambda}(K_{D^{\ast}}) 
        }}
        \prod_{\alpha\in\Delta_{\lambda,+}} \operatorname{mult}(f_{\alpha}) w_{t_{\alpha}}(f_{\alpha}) + O(d^{-1})
    \]
    coefficientwise. Applying Lemma~\ref{lem:zeta_flow} separately to each $\alpha\in\Delta_{\lambda,+}$ gives the second assertion.
\end{proof}

Put
\[
    2\rho := \sum_{\alpha\in\Delta_{+}}\alpha,\qquad
    L := \prod_{\alpha\in\Delta_{+}} t_{\alpha}^{-\frac{\operatorname{rot}(D^{\ast})}{2}}.
\]

\begin{prop}[Contribution of caps and cups]
    \label{prop:semisimple-cap-cup}
    One has
    \[
        Z_{d,\lambda}^{RT} = LZ_{d,\lambda}^{A} + O(d^{-1})
    \]
    coefficientwise.
\end{prop}

\begin{proof}
    As in \S\ref{subsec:semisimple-state-sum}, the two segments adjacent to a cap or cup must carry the same label $b\in\mathcal{B}_{d,\lambda}$. If $b$ has weight $\mu$, then the scalar contributions for $\lcap$, $\lcup$, $\rcap$, and $\rcup$ are, respectively, $q^{(2\rho,\mu)}$, $q^{-(2\rho,\mu)}$, $1$, $1$. 

    Fix $m\geq0$. By Lemma~\ref{lem:semisimple-order}, it suffices to consider summands $W_{d,\lambda}^{\mathrm{cross}}(\eta)$ with $r(\eta)\leq m$. For such $\eta$, the proof of that lemma shows that, in every element of $\mathcal{B}_{d,\lambda}$ assigned to an edge of $\Gamma_{T}$, there are at most $2m$ indices $a\in I_{d}$ for which the vector in the tensor factor indexed by $a$ is different from $v_{\lambda}$.    
    Since $\mathcal{B}_{\lambda}$ is finite, the weight of such an element has the form
    \[
        d\lambda-\gamma,
    \]
    where $\gamma$ ranges over a finite set depending on $m$ but not on $d$. Hence
    \[
        q^{\pm(2\rho,d\lambda-\gamma)}
        = q^{\pm d(2\rho,\lambda)} \left(1+O(d^{-1})\right)
    \]
    uniformly in coefficients of degree at most $m$.
    Since the number of caps and cups is fixed, and
    \[
        q^{d(2\rho,\lambda)} = \prod_{\alpha\in\Delta_{+}}t_{\alpha}^{-1/2}, \qquad
        \#\lcap-\#\lcup = \operatorname{rot}(D^{\ast}),
    \]
    the product of the cap and cup contributions is $L\left(1+O(d^{-1})\right)$ uniformly in coefficients of degree at most $m$.
    Lemmas~\ref{lem:semisimple-order} and~\ref{lem:semisimple-index-count}, together with $s(\eta)\leq r(\eta)$, show that the sum of $|W_{d,\lambda}^{\mathrm{cross}}(\eta)|_{(m)}$ over the summands with $r(\eta)\leq m$ is $O(1)$. Therefore
    \[
        Z_{d,\lambda}^{RT} = LZ_{d,\lambda}^{\mathrm{cross}} + O(d^{-1})
    \]
    in coefficients of degree at most $m$.
    Propositions~\ref{prop:semisimple-root-reduction} and~\ref{prop:semisimple-cartan-redistribution} give
    \[
        Z_{d,\lambda}^{\mathrm{cross}} = Z_{d,\lambda}^{A} + O(d^{-1}).
    \]
    Since $L$ is independent of $d$, the result follows in coefficients of degree at most $m$. Since $m$ is arbitrary, the result follows coefficientwise.
\end{proof}

\begin{thm}[Melvin--Morton--Rozansky conjecture for $\mathcal{U}_{q}(\mathfrak{g})$]\label{thm:mmr-semisimple}
    For every dominant integral weight $\lambda$,
    \[
        \lim_{d\to\infty}
        \left.
        \frac{J_{\mathfrak{g}}(K,V_{d\lambda})}
             {\operatorname{dim}_{q}(V_{d\lambda})}
        \right|_{q=e^{h/d}}
        = \prod_{\alpha\in\Delta_{+}} \Delta_{K}\!\left(e^{-2(\lambda,\alpha)h}\right)^{-1}.
    \]
    The limit is coefficientwise in $h$.
\end{thm}

\begin{proof}
    If $\alpha\in\Delta_{+}\setminus\Delta_{\lambda,+}$, then $\lambda_{\alpha}=0$, and hence $(\lambda,\alpha)=0$, $t_{\alpha}=q^{-2d(\lambda,\alpha)}=1$.
    Thus
    \[
        \zeta(K_{D^{\ast}},w_{t_{\alpha}})
        =\zeta(K_{D^{\ast}},w_{1})
        =1.
    \]
    Taking $d\to\infty$ coefficientwise in Propositions~\ref{prop:semisimple-cap-cup} and~\ref{prop:semisimple-flow-evaluation},
    \begin{align*}
        \lim_{d\to\infty}
        \left.
        \frac{J_{\mathfrak{g}}(K,V_{d\lambda})}
             {\operatorname{dim}_{q}(V_{d\lambda})}
        \right|_{q=e^{h/d}}
        &=
        \left(
            \prod_{\alpha\in\Delta_{+}}
            t_{\alpha}^{-\frac{\operatorname{rot}(D^{\ast})}{2}}
        \right)
        \left(
            \prod_{\alpha\in\Delta_{\lambda,+}}
            \zeta(K_{D^{\ast}},w_{t_{\alpha}})
        \right)\\
        &=
        \left(
            \prod_{\alpha\in\Delta_{+}}
            t_{\alpha}^{-\frac{\operatorname{rot}(D^{\ast})}{2}}
        \right)
        \left(
            \prod_{\alpha\in\Delta_{\lambda,+}}
            \zeta(K_{D^{\ast}},w_{t_{\alpha}})
        \right)
        \left(
            \prod_{\alpha\in\Delta_{+}\setminus\Delta_{\lambda,+}}
            \zeta(K_{D^{\ast}},w_{t_{\alpha}})
        \right)\\
        &= \prod_{\alpha\in\Delta_{+}}
            t_{\alpha}^{-\frac{\operatorname{rot}(D^{\ast})}{2}}\zeta(K_{D^{\ast}},w_{t_{\alpha}})\\
        &= \prod_{\alpha\in\Delta_{+}}\Delta_{K}(e^{-2(\lambda,\alpha)h})^{-1}.
    \end{align*}
\end{proof}

\bibliographystyle{amsalpha}
\bibliography{References}

\end{document}